\documentclass[11pt]{amsart}
\usepackage[left=1.20in, right=1.20in, top=1.4in, bottom=1.26in]{geometry} 
\usepackage[dvipsnames]{xcolor}
\usepackage[shortlabels]{enumitem}
\usepackage{mathrsfs}
\usepackage{amsfonts}
\usepackage{amsmath}
\usepackage{amssymb}
\usepackage{amsthm}
\usepackage{stmaryrd}
\usepackage{textcomp}
\usepackage{mathtext}
\usepackage{yfonts}
\usepackage{tikz}
\usepackage{sidecap}
\usetikzlibrary{matrix,arrows}

\usepackage{comment}
\usepackage{hyperref}

\let\originalleft\left
\let\originalright\right
\renewcommand{\left}{\mathopen{}\mathclose\bgroup\originalleft}
\renewcommand{\right}{\aftergroup\egroup\originalright}

\usepackage{float}
\restylefloat{table}
\restylefloat{figure}

\makeatletter
\newcommand\suchthat{\@ifstar
  {\mathrel{}\middle\vert\mathrel{}}
  {\mid}}
\makeatother

\theoremstyle{plain}
\newtheorem{theorem}{Theorem}[section]
\newtheorem{lemma}[theorem]{Lemma}

\newtheorem{proposition}[theorem]{Proposition}

\newtheorem{corollary}[theorem]{Corollary}

\theoremstyle{definition}
\newtheorem{example}[theorem]{Example}
\newtheorem{remark}[theorem]{Remark}

\usepackage{nicefrac}
\newcommand{\niceonehalf}{{\nicefrac{1}{2}}}

\newcommand{\supp}{{\operatorname{supp}}}

\newcommand{\inv}{{-1}}

\newcommand{\equivalent}{ \Longleftrightarrow }

\DeclareMathOperator{\dist}{\Delta}
\newcommand{\Lip}{\mathsf{L}}
\newcommand{\LipInv}{\mathsf{iL}}
\newcommand{\biLip}{\mathsf{bL}}
\newcommand{\Star}{\operatorname{St}}
\newcommand{\Ball}{\calB}
\newcommand{\mytent}{\sharp}

\newcommand{\bbL}{{\mathbb L}}

\newcommand{\bbN}{{\mathbb N}}

\newcommand{\bbQ}{{\mathbb Q}}
\newcommand{\bbR}{{\mathbb R}}

\newcommand{\calA}{{\mathcal A}}
\newcommand{\calB}{{\mathcal B}}
\newcommand{\calC}{{\mathcal C}}

\newcommand{\calO}{{\mathcal O}}

\newcommand{\calS}{{\mathcal S}}

\newcommand{\calU}{{\mathcal U}}
\newcommand{\calV}{{\mathcal V}}
\newcommand{\calW}{{\mathcal W}}

\newcommand{\calY}{{\mathcal Y}}

\begin{document}

\title[Constructing collars]{Constructing collars in paracompact Hausdorff spaces and Lipschitz estimates}

\author{Martin W.\ Licht}
\address{\'Ecole~Polytechnique~F\'ed\'erale~de~Lausanne (EPFL), 1015 Lausanne, Switzerland}
\email{martin.licht@epfl.ch}


\date{......}


\begin{abstract}
    We give a constructive proof for the following new collar theorem:
    every locally collared closed set that is paracompact in a Hausdorff space is collared. 
    This includes the important special case of locally collared closed sets in paracompact Hausdorff spaces. 
    Importantly, we use Stone's result that every open cover of a paracompact space 
    has an open locally finite refinement which is the countable union of discrete families.
    Furthermore, in the LIP category, our construction yields collars that are locally bi-Lipschitz embeddings. 
    If the initial data satisfy uniform estimates, 
    then this collar is even bi-Lipschitz onto its image and we explicitly bound the constants. 
    We also provide partitions of unity whose Lipschitz constants are bounded by the Lebesgue constant and the order of the cover.
\end{abstract}

\maketitle

\section{Introduction}

Collars are an important concept of geometric topology used widely throughout mathematics. 
For example, the boundaries of manifolds are collared, and embedded two-sided surfaces of codimension one admit two-sided collars. 
Constructing collars is instrumental in proving the Schoenflies theorem in differential topology, and 
collars also appear in the analysis of partial differential equations for extending functions beyond rough domains. 
Past research efforts have brought forth the following \emph{collar theorem}, applicable under various additional assumptions: 
\emph{a subspace is collared if it is locally collared}.

This manuscript discusses a collar theorem that is both general and constructive. 
On the one hand, our collar construction is general enough to collar any locally collared subset \emph{paracompact in} an ambient Hausdorff space. 
On the other hand, it is sufficiently explicit to enable estimates of Lipschitz constants 
when the construction is carried out in Lipschitz topology. 

For the purpose of this manuscript,
a collar along a closed subset $B$ of a topological space $X$ is an embedding $c : B \times [0,1] \to X$ such that $c(x,0) = x$ for all $x \in B$
and such that $c( B \times [0,1) ) \subseteq X$ is an open neighborhood of $B$. 
We then say that $B$ is \emph{collared} in $X$. 
Conceptually, a collar provides a partial coordinate system along $B$ in one single variable.
The fundamental insight is that in many classes of topological spaces,
the existence of a collar along some subspace already follows if such collars are merely known to exist locally along that subspace. 
\\

We contextualize our methods and review the literature on the collar theorem.
We prove that a locally collared closed subset $B$ of a Hausdorff space $X$ is collared 
if it satisfies the condition of being \emph{paracompact in $X$},
meaning that every cover of $B$ by sets open in $X$ has an open locally finite refinement that covers $B$. 
We emphasize that paracompactness of $B$ is a weaker condition and generally not sufficient for the collar theorem~\cite[Example~3.5]{baillif2022collared}. 
That being said, it is not necessary that the ambient space $X$ be paracompact.

Our collar construction relies on a lesser known result by Stone: 
every open cover of a paracompact Hausdorff space has an open locally finite refinement that is a countable union of discrete open families~\cite{stone1948paracompactness}. 
In the light of Stone's result, we first prove that discrete families of collared sets have a collared union,
and then complete the discussion with an induction argument. 

There are essentially two known approaches to proving the collar theorem in non-smooth categories. 
Brown~\cite{brown1962locally} proves a collar theorem for metric spaces, 
relying on the following feature of paracompact spaces, discovered by Michael~\cite{michael1954local}:
if a collection of open sets is closed under taking open subsets, disjoint unions, and finite unions,
then the union of the collection is contained in the collection itself.
Brown proves the collar theorem by showing that the property of \emph{being collared} satisfies those three requirements.
The strategy seems inherently non-constructive.
However, our approach can also be reinterpreted as a variation of Brown's idea:
we show that the class of collared sets is closed under 
taking open subsets, discrete unions, and countable locally finite unions. 

Connelly~\cite{connelly1971new} proves the collar theorem for compact subspaces of Hausdorff spaces
with a different, inherently constructive and geometric method. 
In that approach, one first extends $X$ by glueing a product set $B \times [0,1)$ along $B \subseteq X$.
Next, one stepwise constructs the desired collar by using the local collars 
to pull the original topological space into the extended space.
Connelly only targets compact subspaces, and so the procedure finishes after a finite number of steps. 
En passent, Connelly also deviates from Brown's premises in defining collars to be closed embeddings; this difference has effects when the ambient topology is non-normal~\cite{baillif2022collared}.

Connelly's original article alludes to a possible generalization to the case where $B$ is strongly paracompact,
but no further details are given. 
Let us recall that a topological space is called \emph{strongly paracompact} 
if every open cover admits an open refinement 
such that every member of the refinement intersects only finitely many other members of the refinement.
All strongly paracompact spaces are paracompact but not vice versa. 
How to extend Connelly's result along that line of thought is elaborated in Gauld's textbook on non-metrizable manifolds~\cite[Theorem~3.10]{gauld2014non} and in a contribution by Baillif~\cite[Theorem~2.5]{baillif2022collared}. 
Our approach in this manuscript can be seen as a follow-up to these works.
The latter reference emphasizes that strong paracompactness of the locally collared set is not sufficient, while strong paracompactness of the ambient space is not necessary; indeed, the natural assumption is that the locally collared subset is \emph{strongly paracompact in} its ambient space.
We build upon that line of thought in the present exposition. 

Finite-dimensional manifolds are metrizable if and only if they are paracompact if and only if they are strongly paracompact~\cite[Theorem~2.1]{gauld2014non}. 
However, some metric spaces are not strongly paracompact, let alone compact, and some strongly paracompact spaces are not metrizable. That means that the contributions by Brown and by Connelly (and the latter's extensions by Gauld and Baillif) 
ultimately address different classes of spaces, neither class being a subset of the other. 
This is why we study a collar construction whose scope includes the class of paracompact Hausdorff spaces, even non-metrizable ones.
\\

After studying the collar theorem in the topological category under very general assumptions,
we turn our attention to a special case that offers much more structure: 
the second part of this work addresses collars in Lipschitz topology. 

Does the collar theorem still hold when all spaces are metric and all mappings are locally Lipschitz?
Luukkainen and V\"ais\"al\"a~\cite{luukkainen1977elements} have proven the collar theorem in the LIP category: any set locally collared by locally bi-Lipschitz local collars has a global collar that is locally bi-Lipschitz.
Their proof is similar in spirit to Brown's approach. 
We provide an alternative proof, when the base set is closed, by re-analyzing our collar construction.

In addition, we highlight sufficient conditions for the global collar being bi-Lipschitz. 
On the one hand, this new result is qualitative: 
the global collar is bi-Lipschitz if we start with a collection of local collars that satisfy several uniform estimates.
Those uniform estimates include, e.g., bounds on the order and Lebesgue number of the base cover. 
On the other hand, our quantitative results are explicit estimates of the Lipschitz constants in terms of a few problem parameters. 
We are aware of work by Ghiloni~\cite{ghiloni2010complexity},
who addresses quantitative aspects of Lipschitz collars specifically of compact spaces;
since we include non-compact collared sets in our discussion, we require structurally different techniques. 

As an auxiliary result of independent standing, 
we construct Lipschitz partitions of unity whose Lipschitz constants are bounded in terms of the cover's Lebesgue number and order. The bounds slightly improve estimates in prior publications~\cite{buyalo2008hyperbolic}.
\\

Two-sided collars are also known as \emph{bicollars}. 
One easily glues two one-sided collars in the topological category to build a bicollar;
more care seems imperative in the Lipschitz setting, as we shall discuss. 

It is an intriguing question whether our collar construction works in other settings with additional structure,
such as in the uniform category and in the category of PL manifolds. 
A qualitative transition seems to take place once we consider $C^{1}$ submanifolds or even smoother spaces:
here, the collar theorem is proven easily via local flows. 
Finally, there are numerous aspects of metric geometry that emerge in the discussion of Lipschitz collars,
such as the control of Lebesgue numbers of open covers. 
We leave these questions to future research endeavors.
\\

The remainder of the manuscript is structured as follows. 
In Section~\ref{section:covers}, 
we review definitions and prove fundamental results on subspaces that are paracompact within their ambient space.  
In Section~\ref{section:collardefinitions}, 
we gather the relevant definitions of collars and prove additional auxiliary lemmas.
In Section~\ref{section:topological}, 
we prove the collar theorem in the paracompact case. 
In Section~\ref{section:stronglyparacompact}, 
we discuss the strongly paracompact case in more detail. 
In Section~\ref{section:bicollar}, 
we briefly review the construction of bicollars. 
In Section~\ref{section:lipschitz}, finally, 
we prove the collar theorem in the Lipschitz setting.

\section{Definitions and auxiliary results}\label{section:covers}

This section collects relevant notions of topological spaces and auxiliary results. 
Engelking's monograph~\cite{engelking1989general} serves as our standard reference for general topology. 

We begin by recalling the definitions of some of the terms which we use.
Let $X$ be a topological space.
A \emph{base} for the topology of $X$ is a collection $\calB$ of open sets in $X$ such that every open set in $X$ is a union of sets from $\calB$.
Two sets $A, B \subseteq X$ are \emph{separated} if there exist open sets $U$ and $V$ with $A \subseteq U$ and $B \subseteq V$ but $U \cap V = \emptyset$.
We call $X$ \emph{Hausdorff} if all distinct single-point sets in $X$ are separated, 
\emph{regular} if every disjoint singleton and closed set are separated, 
and \emph{normal} if any two disjoint closed sets are separated.
One immediately sees that singletons are closed in Hausdorff spaces, and that normal Hausdorff spaces are regular.

\begin{lemma}\label{lemma:closurecondition} 
    Let $X$ be a topological space, $U \subseteq X$, and $y \in X$.
    Then $y \in \overline{U}$ if and only if every open neighborhood of $y$ intersects with $U$. 
\end{lemma}
\begin{proof} 
    If $y$ has an open neighborhood disjoint from $U$, then $y \notin \overline{U}$.
    Conversely, if $y \notin \overline{U}$, then $X \setminus \overline{U}$ is an open neighborhood of $y$ 
    disjoint from $\overline{U}$.
\end{proof}

\begin{lemma}\label{lemma:joiningrelativeopensets}
    Let $A, B$ be closed subsets of a topological space $X$.
    If $U_A \subseteq A$ and $U_B \subseteq B$ are open in $A$ and $B$, respectively, and
    \[
        U_A \cap A \cap B = U_B \cap A \cap B,
    \]
    then $U_A \cup U_B$ is open in $A \cup B$. 
\end{lemma}
\begin{proof}
    Let $x \in U_A \cup U_B$ be arbitrary. 
    If suffices to show that $x$ has a neighborhood contained in $U_A \cup U_B$ and open in $A \cup B$.
    There exist two open sets $O_A, O_B \subseteq X$ satisfying $U_A = A \cap O_A$ and $U_B = B \cap O_B$.
    
    If $x \notin B$, 
    then $O_A \setminus B$ is a neighborhood of $x$ that is open in $X$ and $A \cap (O_A \setminus B)$ is both open in $A \cup B$ and a subset of $U_{A}$.
    If $x \notin A$, 
    then a similar argument works.

    Finally,
    if $x \in A \cap B$, then let $O := O_{A} \cap O_{B}$. 
    Obviously, $A \cap O \subseteq U_{A}$ and $B \cap O \subseteq U_{B}$. 
    We know $x \in U_{A} \cap U_{B}$.
    Thus, $(A\cup B) \cap O$ satisfies the desired properties:
    it is open in $A \cup B$, contains $x$, and is a subset of $U_A \cup U_B$. 
\end{proof}

\begin{lemma}\label{lemma:pastinglemma}[Pasting Lemma]
    Let $X$ and $Y$ be topological spaces and let $A, B \subseteq X$ be closed with $X = A \cup B$. 
    If $f : X \rightarrow Y$ is such that $f_{|A}$ and $f_{|B}$ are continuous, then so is $f$.
\end{lemma}
\begin{proof}
    See~\cite[Proposition~2.1.13]{engelking1989general}.
\end{proof}

Let $\calU$ be a family of subsets of $X$.
We call $\calU$ \emph{star-finite} if every member of $\calU$ intersects at most finitely-many other members.
We call $\calU$ \emph{locally finite} if each $x \in X$ has an open neighborhood intersecting only finitely many members of $\calU$.
We call $\calU$ a \emph{cover} if its union equals $X$.
More generally, we call $\calU$ a cover of a set $B \subseteq X$ if the union of $\calU$ includes $B$. 

The following observation is of great utility.

\begin{lemma}
    Let $X$ be a topological space and $\calU$ be a locally finite family of subsets of $X$.
    Then $\left\{ \overline U \suchthat* U \in \calU \right\}$ is locally finite, too.
    Moreover, 
    \[
        \overline{\bigcup \calU} = \bigcup \left\{ \overline U \suchthat* U \in \calU \right\}.
    \]
\end{lemma}
\begin{proof}
    See~\cite[Theorem~1.1.11, Theorem~1.1.13]{engelking1989general}.
\end{proof}

\begin{lemma} 
    If $X$ is a normal space and $(U_{\alpha})_{\alpha \in \kappa}$ is a locally finite open cover of $X$,
    then $X$ has an open cover $(V_{\alpha})_{\alpha \in \kappa}$ satisfying $\overline{V_{\alpha}} \subseteq U_{\alpha}$ for all $\alpha \in \kappa$. 
\end{lemma}
\begin{proof}
    See~\cite[Theorem~1.5.18]{engelking1989general}.
\end{proof}

A topological space $X$ is called \emph{paracompact} if every open cover of that space has a locally finite refinement.
We call $X$ \emph{strongly paracompact} if every open cover of that space has a star-finite refinement.

\begin{remark}\label{remark:metrizablevsstronglyparacompact}
    Metrizable spaces are paracompact, but the relation between metrizability and strong paracompactness deserves to be commented.
    
    On the one hand, 
    some strongly paracompact spaces are not metrizable. 
    The line with two origins is a (non-Hausdorff) strongly paracompact space that is not metrizable.
    Regular Lindel\"of\footnote{Recall that a space is called \emph{Lindel\"of} if every open cover has a countable subcover.} Hausdorff spaces are strongly paracompact~\cite[Chapter~2.4, p.155]{engelking1995dimension}.
    The Sorgenfrey line is a completely normal (hence regular) Hausdorff space that is Lindel\"of, hence strongly paracompact, but not metrizable.

    On the other hand, 
    some metric spaces are not strongly paracompact.
    A simple example is the Hedgehog space with uncountably many spines;
    indeed, if that space were strongly paracompact, then it were a Lindel\"of metric space and hence separable, which it is not.
    Another example is known as Roy's space, which is a non-separable completely metrizable space that is not strongly paracompact,
    and the reader is referred to Engelking's book on dimension theory~\cite[Theorem~4.1.4, Remark 4.1.6]{engelking1995dimension} for a brief outline.
    However, separable metric spaces are strongly paracompact since they are regular and Lindel\"of.
\end{remark}

In what follows, we adhere to the following informal idea:
if there is any \textit{property P} for a topological space $X$, 
then we analogously introduce for any subspace $B \subseteq X$ the \textit{property P in X}. 
The discussion below will illustrate that many implications between properties of topological spaces 
lead to analogous implications of properties of a subspace within its ambient space.\footnote{Which statements on properties of topological spaces have an ``ambient'' analogue is an intriguing metamathematical question but outside of this article's scope.}

Suppose that $B \subseteq X$ is a subset.
We say that $B$ is \emph{paracompact in $X$}
if every cover $\calU$ of $B$ by sets open in $X$ has a locally finite refinement 
that is a cover of $B$ by sets open in $X$.
We say that $B$ is \emph{strongly paracompact in $X$}
if every cover $\calU$ of $B$ by sets open in $X$ has a star-finite refinement 
that is a cover of $B$ by sets open in $X$.
Clearly, every (strongly) paracompact space is (strongly) paracompact in itself. 

\begin{lemma}\label{lemma:paracompactnessofsubsets}
    Let $X$ be a topological space and let $B$ be a closed subset.
    \begin{enumerate}[1., wide=10pt, itemindent=\parindent, leftmargin=0pt, topsep=0pt, itemsep=0pt]
    \item If $X$ is (strongly) paracompact, then $B$ is (strongly) paracompact in $X$.
    \item If $X$ is (strongly) paracompact in $X$, then $B$ is (strongly) paracompact.
    \end{enumerate}
\end{lemma}
\begin{proof}
    Suppose that $X$ is (strongly) paracompact. 
    Let $\calU$ be a cover of $B$ by sets open in $X$.
    Then $\calU \cup \{ X\setminus B \}$ is an open cover of $X$
    and has a locally finite (star-finite) open refinement $\calV$. Then
$\calV$
    is a locally finite (star-finite) open cover of $B$ by sets open in $X$. 
    If we remove all subsets of $X \setminus B$ from that cover, 
    then it is a refinement of $\calU$.
    
    Suppose that $B$ is (strongly) paracompact in $X$ and let $\calU$ be an open cover of $B$
    by sets relatively open in $B$.
    Every $U \in \calU$ is the restriction of a set $U'$ open in $X$,
    and these sets form an open cover $\calU'$ of $B$ by sets open in $X$.
    We let $\calV'$ be an open locally finite (star-finite) refinement.
    Restricting the members of $\calV'$ to the set $B$
    produces a cover $\calV$ of $B$ by sets relatively open in $B$
    that is locally finite (star-finite) and a refinement of $\calU$.
\end{proof}

\begin{example}
    Let $\bbL$ denote the \emph{closed long ray}, which is a normal Hausdorff space. 
    If $X$ is any strongly paracompact set, then $\{0\} \times X$ is strongly paracompact in $\bbL \times X$. 
    However, $\bbL$ is not paracompact.
\end{example}

A family $\calU$ of sets is called \emph{discrete} if its members have pairwise disjoint closures.\footnote{Some authors additionally require that arbitrary unions of those closures are closed again. However, the union of a locally finite family of closed sets is closed, and hence that additional requirement is redundant.}
A family $\calU$ that is the countable union of discrete families is called \emph{$\sigma$-discrete}.
$X$ is called \emph{collectionwise normal} if whenever $( A_\alpha )_{\alpha \in \kappa}$ is a discrete family of closed sets in $X$,
then there exists a disjoint open family $( O_\alpha )_{\alpha \in \kappa}$ such that $A_\alpha \subseteq O_\alpha$ for each $\alpha \in \kappa$.
Any subspace $B \subseteq X$ is called \emph{collectionwise normal in $X$} if whenever $( A_\alpha )_{\alpha \in \kappa}$ is a discrete family of closed sets in $B$, then there exists a disjoint open family $( O_\alpha )_{\alpha \in \kappa}$ of sets open in $X$ such that $A_\alpha \subseteq O_\alpha$ for each $\alpha \in \kappa$.

\begin{lemma}\label{lemma:paracompactnesswithsameindexset}
    Let $B$ be paracompact in the topological space $X$.
    If $\calU = (U_\alpha)_{\alpha\in\kappa}$ is a cover of $B$, 
    then there exists a locally finite cover $\calW = (W_\alpha)_{\alpha\in\kappa}$ of $B$
    such that $W_\alpha \subseteq U_\alpha$ for all $\alpha \in \kappa$. 
\end{lemma}
\begin{proof}
    Since $B$ is paracompact in $X$, there exists a locally finite family of open sets 
    $\calU' = (U'_{\alpha'})_{\alpha'\in\kappa'}$,
    indexed over some set $\kappa'$, 
    that covers $B$ and such for all $\alpha' \in \kappa'$ 
    there exists $\alpha(\alpha') \in \kappa$ with $U'_{\alpha'} \subseteq U_{\alpha(\alpha')}$. 
    We define 
    \begin{align*}
        W_{\alpha} := \bigcup \left\{ U'_{\alpha'} \in \calU' \suchthat* \alpha(\alpha') = \alpha \right\}.
    \end{align*}
    Then $(W_\alpha)_{\alpha\in\kappa}$ is a locally finite cover of $B$ by sets open in $X$
    such that $W_\alpha \subseteq U_\alpha$ for all $\alpha \in \kappa$. 
    The proof is complete. 
\end{proof}

\begin{remark}
The preceding lemma has no analogue for the notion of strong paracompactness.
    For example, take the real line $\bbR$ and its cover $\calU$ 
    by the open set $U_0 := \bbR \setminus \bbN$ and the intervals 
    $U_{i} = (i-\epsilon,i+\epsilon)$, $i \in \bbN$,
    where $\epsilon > 0$ is, say, $\epsilon = 0.01$.
    The real line is strongly paracompact, being a separable metric space. 
    If $\calV = (V_{i})_{i\in\bbN_0}$ is an open refinement such that $V_{i} \subseteq U_{i}$ for all $i \in \bbN_{0}$,
    then $\calV$ cannot be star-finite. 
\end{remark}

\begin{lemma}\label{lemma:paracompactimpliescollectionwisenormal}
    Let $B$ be a closed subspace of a Hausdorff space $X$.
    If $B$ is paracompact in $X$, then $B$ is collectionwise normal in $X$.
\end{lemma}
\begin{proof} 
    We develop the proof in several steps.
    \begin{enumerate}[1., wide=10pt, itemindent=\parindent, leftmargin=0pt, topsep=0pt, itemsep=0pt]
    \item 
    Let $\calA = ( A_\alpha )_{\alpha \in \kappa}$ be a discrete family of closed sets in $B$.
    Since $B$ is closed, its closed subsets are closed in $X$ too. 
    We introduce the closed sets 
    \begin{align*}
        B_\alpha := \bigcup_{ \substack{ \beta \in \kappa \\ \alpha \neq \beta } } A_\beta,
        \quad
        \alpha \in \kappa
        .
    \end{align*}

    \item Let $\alpha \in \kappa$ and let $x \in A_\alpha$, $y \in B_\alpha$.
    Then $x \neq y$ because $A_\alpha \cap B_\alpha = \emptyset$. 
    Since $X$ is Hausdorff, 
    $x$ and $y$ have disjoint open neighborhoods $O_{\alpha,x,y}$ and $U_{\alpha,x,y}$, respectively. 
    Clearly, $\left\{ U_{\alpha,x,y} \suchthat* y \in B_\alpha \right\}$ is a family of open sets that covers the closed set $B_{\alpha}$. 
Using that $B$ is paracompact in $X$, 
    we find a locally finite family 
$\left\{ U_{\alpha,x,y}' \suchthat* y \in B_\alpha \right\}$
that covers $B_\alpha$ and with $U_{\alpha,x,y}' \subseteq U_{\alpha,x,y}$.
    In particular, $x$ has an open neighborhood $\widetilde O_{\alpha,x}$ that intersects only finitely many members of this family.
    So there exists a finite subset ${P_{\alpha,x}} \subseteq B_\alpha$ satisfying $\widetilde O_{\alpha,x} \cap U_{\alpha,x,y}' = \emptyset$ for all $y \in B_\alpha \setminus {P_{\alpha,x}}$.
    Therefore, 
    \begin{align*}
        O_{\alpha,x} := \widetilde O_{\alpha,x} \cap \bigcap_{y \in {P_{\alpha,x}}} O_{\alpha,x,y}
        ,
        \quad 
        U_{\alpha,x} := \bigcup_{y \in B_\alpha} U_{\alpha,x,y}'
    \end{align*}
    are open neighborhoods of $x$ and $B_{\alpha}$, respectively, and disjoint by construction.

    \item 
    Clearly, $\left\{ U_{\alpha,x} \suchthat* x \in A_\alpha \right\}$ is a family of open neighborhoods of the closed set $B_{\alpha}$.
Using that $B$ is paracompact in $X$, we find a locally finite family 
$\left\{ U_{\alpha,x}' \suchthat* x \in A_\alpha \right\}$
that covers $B_\alpha$ and satisfies $U_{\alpha,x}' \subseteq U_{\alpha,x}$ for each $\alpha \in \kappa$. 
    We introduce another open neighborhood of $B_\alpha$ by 
    \[
        U_{\alpha} := \bigcup_{x\in A_\alpha} U_{\alpha,x}'.
    \] 
    Due to local finiteness, every $x \in A_\alpha$ has an open neighborhood $\widehat O_{\alpha,x}$
    such that there exists a finite ${K_{\alpha,x}} \subseteq A_\alpha$ with the property that 
    $\widehat O_{\alpha,x} \cap U_{\alpha,y}' = \emptyset$ for every $y \in A_\alpha \setminus {K_{\alpha,x}}$.
    We recall that $O_{\alpha,y} \cap U_{\alpha,y} = \emptyset$ for every $y \in {K_{\alpha,x}}$.
    Hence 
    \[
        O_{\alpha} = \bigcup_{x \in A_\alpha} \left( \widehat O_{\alpha,x} \cap \bigcap_{y \in {K_{\alpha,x}}} O_{\alpha,y} \right)
    \]
    is an open neighborhood of $A_\alpha$ that is union of sets disjoint from $U_\alpha$.

    \item 
    For any $\alpha \in \kappa$ we have found disjoint open neighborhoods $O_{\alpha}$ and $U_{\alpha}$ of the closed sets $A_\alpha$ and $B_\alpha$, respectively. 
    For any $\alpha, \beta \in \kappa$: 
    \begin{gather*}\alpha = \beta
        \quad\equivalent\quad  
        A_\alpha \subseteq O_{\beta}
        \quad\equivalent\quad  
        A_\alpha \cap \overline{ O_{\beta} } \neq \emptyset 
        .
    \end{gather*} 
    The family $\calO := \{ O_\alpha \}_{\alpha\in\kappa}$ is an open cover of the closed set $\bigcup \calA$.
Since $B$ is paracompact in $X$, 
    we assume without loss of generality that $\calO$ is locally finite.
    \item 
    Since $\calO$ is locally finite, the sets 
    \begin{align*}
        Y_{\alpha} := O_{\alpha} \setminus \bigcup_{\beta\in \kappa, \beta \neq \alpha} \overline{O_{\beta}}
    \end{align*}
    are open. 
    One easily checks that  
    \[
        \alpha = \beta
        \quad\equivalent\quad  
        A_\alpha \subseteq Y_{\alpha}
        \quad\equivalent\quad  
        A_\alpha \cap \overline{ Y_{\alpha} } \neq \emptyset 
        .
    \]
    Finally, if $\alpha, \beta \in \kappa$ are distinct, then $Y_\alpha \cap Y_{\beta} = \emptyset$,
    as evident from the construction. 
    \end{enumerate}
    The desired family is $\calY := \{ Y_\alpha \}_{\alpha\in\kappa}$, and the proof is complete. 
\end{proof}

\begin{corollary}
    Paracompact Hausdorff spaces are collectionwise normal. 
\end{corollary}

Open covers of paracompact spaces have locally finite refinements. 
Much more is to be told about open covers of paracompact Hausdorff spaces, and this will be critical for our main result. 

Towards that end, we introduce two more definitions.
Given any collection $\calU$ and any set $A \subseteq X$, 
we let $\Star(A,\calU)$ denote the union of all members of $\calU$ that intersect with $A$. 
We call $\calU$ a \emph{star-refinement} of another collection $\calV$ if for every $U \in \calU$ there exists $V \in \calV$ 
such that we have the inclusion $\Star(U,\calU) \subseteq V$.

\begin{lemma}\label{lemma:normalimpliesstarrefinement}
    Every locally finite open cover $\calU$ of a normal space $X$ has an open star-refinement $\calW$.
\end{lemma}
\begin{proof}
    Let $\calU = \left\{ U_{\alpha} \suchthat* \alpha \in \kappa \right\}$ be a locally finite cover of $X$ by open sets,
    indexed over some set $\kappa$.
    Because $X$ is normal and $\calU$ is locally finite,
    there exists another open cover $\calV = \left\{ V_{\alpha} \suchthat* \alpha \in \kappa \right\}$ of $X$
    such that $\overline{V_{\alpha}} \subseteq U_{\alpha}$ satisfying $\alpha \in \kappa$.

    For each $x \in X$, there exists a maximal finite non-empty set $\kappa(x) \subseteq \kappa$ 
    such that $\alpha \in \kappa(x)$ if and only if every relatively open neighborhood of $x$ intersects with $V_{\alpha}$.
    If $\alpha \in \kappa(x)$, then $x \in \overline{V_{\alpha}} \subseteq U_{\alpha}$.
    We fix an open neighborhood $V_x \subseteq X$ of $x$ such that $V_x \cap V_{\alpha} \neq \emptyset$ if and only if $\alpha \in \kappa(x)$.
    We define 
    \begin{align*}
        W_{x} 
        = 
        V_{x} 
        \cap 
        \bigcap_{ \substack{ \alpha \in \kappa(x) \\ x    \in V_{\alpha} } } V_{\alpha} 
        \cap 
        \bigcap_{ \substack{ \alpha \in \kappa(x) \\ x \notin V_{\alpha} } } U_{\alpha} 
        .
    \end{align*}
    Clearly, $W_{x}$ is an open neighborhood of $x$.
    In particular, 
    $\calW = \left\{ W_{x} \suchthat* x \in B \right\}$
    is an open cover of $B$. 
    
    Let $y \in X$ and $\beta \in \kappa$ with $y \in V_{\beta}$. 
    If $x \in B$ with $y \in W_{x}$, then $W_x \subseteq V_{x}$ intersects with $V_{\beta}$, and thus $\beta \in \kappa(x)$.
    In the case $x \in V_{\beta}$, we have $W_{x} \subseteq V_{\beta} \subseteq U_{\beta}$.
    In the case $x \notin V_{\beta}$, we have $W_{x} \subseteq U_{\beta}$ too. 
    In other words, $\Star(y,\calW) \subseteq U_{\beta}$.
    
    Now consider any $z \in B$. 
    There exists $\alpha \in \kappa$ such that $z \in V_{\alpha}$ and thus $W_{z} \subseteq V_{\alpha}$. 
    The argument above now implies that every $y \in W_{z}$ satisfies $\Star(y,\calW) \subseteq U_{\alpha}$. 
    Hence $\Star(W_{z},\calW) \subseteq U_{\alpha}$.
    We conclude that $\calW$ is an open star-refinement of $\calU$. 
\end{proof}

\begin{proposition}\label{proposition:paracompactimpliessigmadiscrete}
    Let $X$ be a paracompact Hausdorff space.
    If $\calU = \left\{ U_{\alpha} \suchthat* \alpha \in \kappa \right\}$ is an open cover of $X$, 
    then $\calU$ has a locally finite open refinement which is a countable union 
    \begin{align*}
        \calV = \bigcup_{i=1}^{\infty} \calV_{i},
    \end{align*}
    where each $\calV_{i} = \left\{ V_{i,\alpha} \suchthat* \alpha \in \kappa \right\}$ is discrete and satisfies
    $V_{i,\alpha} \subseteq U_{\alpha}$ for all $\alpha \in \kappa$.
\end{proposition}
\begin{proof}
    The proof is organized into several steps. 
    \begin{enumerate}[1., wide=10pt, itemindent=\parindent, leftmargin=0pt, topsep=0pt, itemsep=0pt]
    \item 
    Suppose that $\calU = (U_\alpha)_{\alpha\in\kappa}$ is a collection of open sets that cover of ${X}$. 
    Note that $X$ is normal. 
    We set $\calU_0 = \calU$, and recursively choose a star-refinement $\calU_{i+1}$ of $\calU_{i}$.
    Then each $\calU_{i}$ is an open refinement of $\calU$.
    
    \item 
    For $i \in \bbN$ and $\alpha \in \kappa$, we define 
    \begin{gather*}
        U_{i,\alpha}
        :=
        \left\{ x \in X \suchthat* \Star(V,\calU_{i}) \subseteq U_\alpha \text{ for an open neighborhood $V$ of $x$}  \right\}
        .
    \end{gather*}
    If $i \in \bbN$ and $U \in \calU_{i}$,
    then $\Star(U,\calU_{i})$ is a subset of some member of $\calU_{i-1}$,
    and hence there exists $\alpha \in \kappa$ with $\Star(U,\calU_{i}) \subseteq U_{\alpha}$. 
    It is thus evident that $(U_{i,\alpha})_{\alpha\in\kappa}$ covers $X$.
    
    Moreover, for any $i \geq 1$ and any open set $V \subseteq X$ 
    we have $\Star(V,\calU_{i+1}) \subseteq \Star(V,\calU_{i})$, 
    and therefore $U_{i,\alpha} \subseteq U_{i+1,\alpha}$ for every $\alpha \in \kappa$.
    
    \item 
    We assume without loss of generality that $\kappa$ is well-ordered,
    as per the axiom of choice. 
    We define $\calV_{i} := (V_{i,\alpha})_{\alpha\in\kappa}$, $i \geq 1$, where 
    \begin{gather*}
        V_{i,\alpha}
        :=
        U_{i,\alpha} \setminus \overline{ \bigcup_{\substack{ \beta \in \kappa \\ \beta < \alpha }} U_{i+1,\alpha} }
        , \quad 
        \alpha \in \kappa
        .
    \end{gather*}
    We have $V_{i,\alpha} \subseteq U_{i,\alpha} \subseteq U_{\alpha}$ for any $\alpha \in \kappa$.
    We want to show that each $\calV_{i}$ is discrete and that their union is an open cover.

    \item
    As an auxiliary result, we show that all members of $\calU_{i+1}$ that intersect with $U_{i,\alpha}$ are already subsets of $U_{i+1,\alpha}$.
So suppose that $U \in \calU_{i+1}$.
    By assumption, there exists $W \in \calU_{i}$ such that $\Star(U,\calU_{i+1}) \subseteq W$.
    Obviously, $U \subseteq W$, and $W \subseteq \Star(U,\calU_{i})$ because $W \in \calU_{i}$ with $W \cap U \neq \emptyset$.
    
    Let us now assume that $U$ intersects with some $U_{i,\alpha}$, where $\alpha \in \kappa$,
    which means there exists $x \in U \cap U_{i,\alpha}$.
    Every $x \in U_{i,\alpha}$ has an open neighborhood $V_{x,i,\alpha}$ 
    such that $\Star(V_{x,i,\alpha},\calU_{i}) \subseteq U_{\alpha}$.
Hence, $\Star(x,\calU_{i}) \subseteq U_{\alpha}$.
    But we also have $x \in U \subseteq W \in \calU_{i}$, and therefore $W \subseteq \Star(x,\calU_{i})$.
In combination, 
    \[
        \Star(U,\calU_{i+1}) \subseteq W \subseteq \Star(x,\calU_{i}) \subseteq U_{\alpha}.
    \]
    This implies $U \subseteq U_{i+1,\alpha}$, and this proves the auxiliary result. 
    
    \item
    We now prove that each $\calV_{i}$ is discrete. 
    Let $\alpha,\beta \in \kappa$ such that $\alpha < \beta$. 
    Suppose that there exists $z \in \overline{V_{i,\alpha}} \cap \overline{V_{i,\beta}}$.
    Then every open neighborhood of $z$ intersects with $V_{i,\alpha}$ and $V_{i,\beta}$.
    In particular, since $\calU_{i+1}$ is an open cover, there exists $U \in \calU_{i+1}$ that contains $z$ and that intersects both $V_{i,\alpha}$ and $V_{i,\beta}$.
    By definition and using the auxiliary result above, $U$ intersects $U_{i,\alpha}$ and is thus a subset of $U_{i+1,\alpha}$.  
    But then $U_{i+1,\alpha}$ intersects $V_{i,\beta} \subseteq X \setminus U_{i+1,\alpha}$,
    which is a contradiction. 
    We conclude that $\calV_{i}$ is a discrete family. 
    
    \item
    Now we prove the covering property. 
    We recall that $(U_{i,\alpha})_{\alpha \in \kappa}$ is an open cover of $X$ for all $i \geq 1$.
    Hence, for every $x \in X$ there exists a minimal $\alpha(x) \in \kappa$ such that there exists $i \geq 1$ with $x \in U_{i,\alpha(x)}$.
    We let $i(x) \geq 1$ be the minimal index for which $x \in U_{i(x),\alpha(x)}$.

    Since $\calU_{i(x)+2}$ is a cover of $X$, 
    there exists $U \in \calU_{i(x)+2}$ with $x \in U$.
    Now suppose that $U$ intersects $\cup_{\alpha < \alpha(x)} U_{\alpha,i(x)+1}$.
    Then $U \subseteq U_{\alpha,i(x)+2}$ for some $\alpha \in \kappa$ with $\alpha < \alpha(x)$, as by our auxiliary result, which means $\Star(U,\calU_{i(x)+2}) \subseteq U_{\alpha}$.
    But then $x \in U_{\alpha,i(x)+2}$ and $\alpha < \alpha(x)$, contradictory to our assumption of $\alpha(x)$.
    We conclude that $U$ is an open neighborhood of $x$ not intersecting $\cup_{\alpha < \alpha(x)} U_{\alpha,i(x)+1}$,
    meaning 
    \[
        x \notin \overline{\cup_{\alpha < \alpha(x)} U_{\alpha,i(x)+1}}.
    \]
    Since we have already assumed $x \in U_{i(x),\alpha(x)}$, definitions imply that $x \in V_{i(x),\alpha}$. 
    That means that $\calV = (V_{i,\alpha})_{i \in \bbN_{0},\alpha \in \kappa}$ is an open cover of $X$.
    Clearly, it can be written as the countable union $\calV = \bigcup_{i \in \bbN} \calV_{i}$.
    \item 
Finally, we extract a locally finite subcover over the same index set as $B$ is paracompact. 
    \end{enumerate}
    This completes the proof. 
\end{proof}

\begin{remark}
    Lemma~\ref{lemma:normalimpliesstarrefinement} and its proof are adapted from Lemma~5.1.16 in Engelking's monograph~\cite{engelking1989general}.
Proposition~\ref{proposition:paracompactimpliessigmadiscrete} was first shown by Stone~\cite{stone1948paracompactness};
    see Theorem~2, Theorem~1, and the remark after Theorem~1 in that reference.
    We include the proof for technical completeness since this specific result is rarely explicitly stated on its own. 
\end{remark}

\section{Notions of Collars}\label{section:collardefinitions}

In what follows, $X$ is a topological space and $B \subseteq X$ is a closed subset, equipped with the subspace topology. 

A \emph{local collar} of $B$ in $X$ is an embedding $c : \overline{U} \times [0,1] \to X$,
where $U$ is a relatively open subset of $B$,
such that $c(x,0) = x$ for all $x \in \overline{U}$,
such that $c^{-1}(B) = \overline{U} \times \{0\}$,
and such that $c( U \times [0,1) )$ is open in $X$.
We call $\overline{U}$ the \emph{base} of the local collar. 
A local collar $c : B \times [0,1] \to X$ is called a \emph{collar} of $B$.
We call $B$ \emph{collared} if it has a collar.
We will frequently put to use that any local collar $c : \overline{U} \times [0,1] \to X$ is a collar of the closed set $\overline{U}$. 

We say that $B$ is \emph{locally collared} in $X$
if every point $x \in B$ has got a relatively open neighborhood $U \subseteq B$
such that there exists a local collar $c : \overline{U} \times [0,1] \to X$ of $B$ in $X$.
A \emph{collar cover} of $B$ in $X$ is a collection
\begin{align*}\calC = \left\{ c_{\alpha} : \overline{U_{\alpha}} \times [0,1] \to X \suchthat* \alpha \in \kappa \right\} 
\end{align*}
of local collars of $B$ in $X$, indexed over some set $\kappa$,
such that the collection $\calU = (U_{\alpha})_{\alpha \in \kappa}$ is an open cover of $B$. 
We call $\calU$ the \emph{cover associated with} $\calC$.
Clearly, $B$ is locally collared in $X$ if and only if it has a collar cover in $X$. 
Notice also that for any such collar cover $\calC$, 
the collection of sets 
\begin{align*}
    \calW = \left\{ c_{\alpha}\left( U_{\alpha} \times [0,1)  \right) \suchthat* \alpha \in \kappa \right\}
\end{align*}
is a collection of open subsets of $X$ covering  $B$.

A local collar is called \emph{strong} if it is a closed embedding, that is, its image is closed in $X$.
Analogously, we define the notions of \emph{strong collar} and \emph{strong collar cover} of $B$ in $X$,
as well as the notions of \emph{strongly collared} and \emph{locally strongly collared} in $X$.

\begin{remark}\label{remark:simplifiedcollardefinition}
    Suppose that $c : B \times [0,1] \to X$,
    where $B$ is a closed subset of $X$,
    is an embedding 
    such that $c(x,0) = x$ for all $x \in B$ and such that $c( B \times [0,1) )$ is open in $X$.
    The condition $c^{-1}(B) = B \times \{0\}$ already holds because $B$ is an embedding,
    and thus $c$ is a collar of $B$ in $X$.
\end{remark}

\begin{remark}
Some authors prefer the parameter variable to range over $[0,1)$ rather than $[0,1]$;
    one sees that changing this does not alter whether some set is collared or locally collared. 
\end{remark}

Generally speaking,
whether a (locally) collared subset is also (locally) strongly collared depends on the ambient space. 
We mention two results due to Baillif~\cite[Theorem~2.3 and Theorem~2.4]{baillif2022collared}; see also Examples~3.9 and~3.14 in that reference.

\begin{theorem}\label{theorem:mathieu1}
    If $B$ is a closed locally collared subset of a regular Hausdorff space $X$, then $B$ is locally strongly collared in $X$. 
\end{theorem}

\begin{theorem}\label{theorem:mathieu2}
    Let $B$ be a closed, countably metacompact\footnote{A topological space is called \emph{countably metacompact} if every countable open cover has a point-finite open refinement.}, collared subset of a normal Hausdorff space $X$.
    Then $B$ is strongly collared in $X$. 
\end{theorem}

An important special case of the previous result is this:
if $X$ itself is paracompact, then all its collared closed subsets are strongly collared. 

The remainder of this section lists a few auxiliary results that center around the idea of restricting collars. 

\begin{lemma}\label{lemma:continuouscut}
    Let $X$ be a topological space and let $B \subseteq X$ be a closed subset. 
    Let $c : \overline{U} \times [0,1] \to X$ be a (strong) local collar of $B$ in $X$.
    Let $g : \overline{U} \to (0,1]$ be continuous.
    Then 
    \begin{align*}
        h : \overline{U} \times [0,1] \to X, \quad (x,t) \mapsto c(x,tg(x))
    \end{align*}
    is a (strong) local collar of $B$ in $X$.
\end{lemma}
\begin{proof}
    Consider the mapping 
    \begin{align*}
        G : \overline{U} \times [0,1] \to \overline{U} \times [0,1], \quad (x,t) \mapsto (x,tg(x)).
    \end{align*}
    We set $h = c \circ G$. 
    Clearly, $h(x,0) = x$ for all $x \in \overline{U}$ and $h^{-1}(B) = \overline{U} \times \{0\}$.
    We already know that $c( U \times [0,1) )$ is open in $X$ and that $c$ is an embedding.
    Since $G( U \times [0,1) )$ is open in $U \times [0,1)$, 
    it follows that $h( U \times [0,1) )$ is open in $X$.
    
    $G( \overline{U} \times [0,1] )$ is closed in $\overline{U} \times [0,1]$. 
    Since embeddings with closed image map closed sets onto closed sets, 
    it follows that $h$ is closed if $c$ is closed. 
\end{proof}

The following lemma and its proof are similar to~\cite[Theorem~2.4]{baillif2022collared}. 

\begin{lemma}\label{lemma:cutfunctionthatavoidsclosedset}
    Let $X$ be a topological space, let $B \subseteq X$ be a paracompact closed Hausdorff subspace,
    and let $A \subseteq X$ be a closed set with $A \cap B = \emptyset$.
    If $c : B \times [0,1] \to X$ is a collar of $B$ in $X$,
    then there exists a continuous function $d : B \to (0,1]$ such that
    for all $x \in B$ and $t \in [0,d(x)]$ we have $c(x,t) \notin A$.
\end{lemma}
\begin{proof}
    We write $W = c( B \times [0,1] )$ for the image of the collar. 
    The intersection $W \cap A$ is closed in $W$.
    We write $D := c^{-1}(A \cap W)$ for the preimage of $A$ in $B \times [0,1]$ under $c$.
    Since $c$ is an embedding, $D$ is closed in $B \times [0,1]$.
    
    Since $B$ is paracompact and $[0,1]$ is compact, the product $B \times [0,1]$ is paracompact.
    Moreover, $B \times [0,1]$ is Hausdorff, being the product of two Hausdorff spaces. 
    Hence $B \times [0,1]$ is normal. 
    A fortiori, the disjoint closed sets $B \times \{0\}$ and $D$ are separated by open neighborhoods,
    and in particular $B \times \{0\}$ has an open neighborhood $O$ that is disjoint from $D$.
    We can assume $O \subseteq B \times [0,0.5)$ without loss of generality.
    Recall that 
    \begin{align*}
        \calO := \left\{ V \times I \suchthat* V \subseteq B \text{ open, } I \subseteq [0,1] \text{ open} \right\}
    \end{align*}
    constitutes a basis of open sets for the topology of $B \times [0,1]$.
    Hence, $O$ is the union of members of $\calO$.
    Since $O$ contains $B \times \{0\}$, 
    for every $x \in B$ there exist a relatively open neighborhood $V \subseteq B$ and $t \in (0,0.5]$ 
    such that $(x,0) \in V \times [0,t) \subseteq O$.
    
    In particular,
    for each $x \in B$ we can choose $r_x \in (0,1)$ and an open neighborhood $V_{x} \subseteq B$ 
    such that $x \in V_x$ and $c\left( V_{x} \times [0,r_x] \right) \cap A = \emptyset$.
    The unions of open sets 
    \begin{align*}
        G_{r} := \bigcup\left\{ V_{x} \suchthat* x \in B, r_x = r \right\}, \quad r \in (0,1),
    \end{align*}
    constitute an open cover of $B$.
    
    Since $B$ is paracompact and normal, 
    there exists a locally finite partition of unity $(\lambda_{r} : B \rightarrow [0,1])_{r \in (0,1)}$ subordinate to $\{G_{r}\}_{r \in (0,1)}$. 
    We define 
    \begin{align*}
        d : B \to \bbR, \quad x \mapsto \sum_{r \in (0,1)} r \cdot \lambda_{r}(x).
    \end{align*}
    The function $d$ is well-defined over $B$ because the sum in the definition is always of finitely many terms.
    Its values are in $(0,1]$.
    Moreover, suppose $y \in B$ and let $r \in (0,1)$ be maximal such that $\lambda_{r}(y) > 0$.
    Then there exists $x \in B$ with $y \in V_x$ and 
    $V_{x} \times [0,r_x] \subseteq O$ has empty intersection with $D$. 
    It follows that 
    $c\left( V_{x} \times [0,r_x] \right) \cap A = \emptyset$.
    Hence, $c(x,t) \notin A$ whenever $t \in [0,d(x)] \subseteq [0,r_x]$.
\end{proof}

\begin{corollary}\label{corollary:restrictedcollarthatavoidsclosedset}
    Let $X$ be a topological space and let $B \subseteq X$ be a paracompact closed Hausdorff subset. 
    Let $A \subseteq X$ be a closed set with $A \cap B = \emptyset$.
    If $B$ is (strongly) collared in $X$, then it has a (strong) collar whose image is disjoint from $A$.
\end{corollary}
\begin{proof}
    Given a collar $c : B \times [0,1] \to X$ of $B$ in $X$,
    Lemma~\ref{lemma:cutfunctionthatavoidsclosedset} shows the existence of a continuous function $d : B \to (0,1]$ such that
    $c(x,t) \notin A$ for each $x \in B$ and $t \in [0,d(x)]$.
In accordance to Lemma~\ref{lemma:continuouscut}, we define the collar 
    \begin{align*}
        h : B \times [0,1] \to X, \quad (x,t) \mapsto c(x, d(x) t )
    \end{align*}
    now defines a collar whose image is disjoint from $A$.
    Moreover, if $c$ has closed image in $X$, then $h$ has closed image in $X$.
\end{proof}

We generally cannot assume that the collars map closed sets onto closed sets. 
But as the following result illustrates, 
the collar embedding can only fail to be closed due to its behavior towards the boundary of its domain. 
The proof is not fully trivial because we assume the ambient space merely Hausdorff. 

\begin{theorem}\label{theorem:yinuowangyoumademesmile}
    Let $X$ be a Hausdorff space and let $B \subseteq X$ be a paracompact closed subset. 
    Let $U \subseteq B$ be relatively open and let $c : \overline{U} \times [0,1] \to X$ be a local collar of $B$ in $X$.
    If $D \subseteq \overline{U} \times [0,1]$ is closed and a subset of $U \times [0,1)$,
    then $c(D)$ is closed.
    In particular,
    \begin{align*}
        \overline{c( D )}
        \subseteq 
        c( U \times [0,1) )
        .
    \end{align*}    
\end{theorem}

\begin{proof}
    Since $\overline{U}$ is closed in $B$, it is paracompact and Hausdorff, and so is the product space $\overline{U} \times [0,1]$.
    Since $c$ is a homeomorphism onto its image, 
    the subspace topology of $c( \overline{U} \times [0,1] )$ is paracompact and Hausdorff. 
    In particular, it is normal. 
    
    Every $z \in \partial c( \overline{U} \times [0,1] )$ and $p \in c( D )$
    have disjoint respective neighborhoods $M_{z,p}$ and $N_{z,p}$ open in $X$;
    without loss of generality, $N_{z,p} \subseteq c( U \times [0,1) )$.
    Note that $\{ N_{z,p} \}_{p \in c(D)}$ constitutes an open cover of $c( D )$.
    Since $c(D)$ is closed in the paracompact subspace topology of $c( \overline{U} \times [0,1] )$,
    we assume without loss of generality that $\{ N_{z,p} \}_{p \in c(D)}$ is locally finite.
    That shows that 
    \begin{align*}
        \overline{c( D )}
        \subseteq 
        \overline{ \bigcup_{p \in c(D)} N_{z,p} }
        \subseteq 
        \bigcup_{p \in c(D)} \overline{ N_{z,p} }
        .
    \end{align*}
    By assumption, however, the latter set does not contain $z$. 
    That means that $\overline{c( D )}$ is disjoint from $\partial c( \overline{U} \times [0,1] )$.
    However, since $\overline{c( D )}$ is a subset of the closure of $c( \overline{U} \times [0,1] )$,
    it must be in the interior of the image of $c$.
    
    Suppose that $w \in c( \overline{U} \times [0,1] ) \setminus c(D)$.
    Since the subspace topology of $\overline{U} \times [0,1] )$ is normal, 
    $c(D)$ and $w$ have disjoint relatively open neighborhoods $M$ and $N$, respectively, in the image of $c$.
    We may assume that $M \subseteq c( U \times [0,1) )$ and that $N$ is the restriction of an open neighborhood $N'$ of $w$.
    But that implies $w \notin \overline{c( D )}$.
\end{proof}

\section{Collar construction in paracompact spaces}\label{section:topological}

In this section, we constructively prove the collar theorem for closed sets that are locally collared and paracompact in an ambient Hausdorff space. 
The collar theorem within paracompact Hausdorff spaces is an important special case. 
As a first step towards our main result, we show that a collection of collars with discrete bases can be restricted and joined into a combined local collar.    

\begin{lemma}\label{lemma:collaringdiscrete}
    Let $B$ be closed and paracompact in the Hausdorff space $X$.
    If
    \begin{align*}
        \calC = \left( c_{\alpha} : \overline{U_{\alpha}} \times [0,1] \to X \right)_{\alpha \in \kappa} 
    \end{align*}
    is a family of (strong) local collars of $B$ in $X$ such that the family $(U_{\alpha})_{\alpha \in \kappa}$ is discrete, 
    then $B$ has a (strong) local collar
    \begin{align*}
        c : \bigcup_{\alpha \in \kappa} \overline{U_\alpha} \times [0,1] \to X.
    \end{align*}
\end{lemma}
\begin{proof}
    Since $B$ is closed in $X$, its relatively closed subsets are closed in $X$ too. 
    Since $B$ is paracompact in $X$, it is collectionwise normal in $X$,
    so there exists a pairwise disjoint family $\calO = (O_{\alpha})_{\alpha \in \kappa}$ of sets open in $X$
    for which 
    \begin{align*}
        \overline{U_{\alpha}} \subseteq O_{\alpha}, \quad \alpha \in \kappa.
    \end{align*}
    Any set relatively open in $\overline{U_{\alpha}}$ is the restriction 
    of a set relatively open in $O_{\alpha} \cap B$, which itself is relatively open in $B$.
    Consequently, the disjoint union topology and the subset topology on $\bigcup_{\alpha \in \kappa} \overline{U_\alpha}$ agree.\footnote{Moreover, taking the disjoint union commutes with taking the product with $[0,1]$.}
    
    With Corollary~\ref{corollary:restrictedcollarthatavoidsclosedset}, we fix local collars
    \begin{align*}
        c'_{\alpha} : \overline{U_{\alpha}} \times [0,1] \rightarrow X, \quad \alpha \in \kappa,
    \end{align*}
    such that the closure of the image of $c'_{\alpha}$ lies in $O_{\alpha}$.

Every set relatively open in $c'_{\alpha}( \overline{U_{\alpha}} \times [0,1] )$
    is the restriction of a relatively open subset of $O_{\alpha}$,
    which itself is open in $X$. 
    Consequently, the disjoint union topology and the subset topology on the set 
$\bigcup_{\alpha \in \kappa} c'_{\alpha}( \overline{U_{\alpha}} \times [0,1] )$
agree.
    This set is the image of the mapping from $\bigcup_{\alpha \in \kappa} \overline{U_{\alpha}} \times [0,1]$ into $X$ given by 
    \begin{align*}
        c := \bigcup_{\alpha \in \kappa} c'_\alpha.
    \end{align*}
    It is now clear that this is an embedding. 
Next, $c\left( U_{\alpha} \times [0,1) \right) = c'_{\alpha}\left( U_{\alpha} \times [0,1) \right)$ are open subsets of $X$, and so is their union $c\left( \cup_{\alpha\in\kappa} U_\alpha \times [0,1) \right)$.
    Consequently, $c$ is a local collar of $B$ in $X$.
    
    Finally, consider the case that each $c_{\alpha}$ is strong.
    Then each $c'_{\alpha}$ is strong, too, 
    and since the local collars $c'_\alpha$ have disjoint images, 
    $c$ is strong as well.
\end{proof}

In the next step, we show that a closed subset is collared 
if we have a countable locally finite collar cover.

\begin{theorem}\label{theorem:collaring:locallyfinitecountable}
    Let $B$ be a normal closed subset of a Hausdorff space $X$.
    If 
    \begin{align*}
        \calC = \left( c_{i} : \overline{U_{i}} \times [0,1] \to X \right)_{i \in \bbN} 
    \end{align*}
    is a collar cover of $B$ such that $\calU = (U_{i})_{i \in \bbN}$ is locally finite, 
    then $B$ is collared.
\end{theorem}
\begin{proof} 
    The proof is organized into the following steps.
    \begin{enumerate}[1., wide=10pt, itemindent=\parindent, leftmargin=0pt, topsep=0pt, itemsep=0pt]
    \item Since $B$ is normal and $\calU$ is locally finite, 
    there exists a continuous partition of unity $( \lambda_{i} : B \rightarrow \bbR )_{i \in \bbN}$ subordinate to the cover $\calU$. 
    Each $\lambda_{i} : B \to [0,1]$ is non-zero over an open set $V_{i}$ satisfying $\overline{V_{i}} \subseteq U_{i}$.
    Notice that $\calV = \left\{ V_{i} \suchthat* i \in \bbN \right\}$ is a locally finite open cover of $B$.

    \item 
    For any $i \in \bbN$, 
    we define the two auxiliary sets 
    \begin{align*}
Q_{i,\mytent} &:= \left\{ (x,t) \in B \times [0,1] \suchthat* t \leq \frac 1 2 \lambda_{i}(x) \right\}, 
        \\
        O_{i,\mytent} &:= \left\{ (x,t) \in B \times [0,1] \suchthat* t <    \frac 1 2 \lambda_{i}(x) \right\}.
    \end{align*}
    Note that $Q_{i,\mytent}$ is closed and $O_{i,\mytent}$ is open in $B \times [0,1]$, respectively.

    \item For each $i \in \bbN$, we introduce the homeomorphism 
    \begin{align*} 
        \Xi_{i} : \overline{U_{i}} \times [0,1] \to \overline{U_{i}} \times [0,1]
    \end{align*}
    defined by 
    \begin{align*} 
        \Xi_{i}(x,t)
        = 
        \begin{cases} 
            \left( x, \frac {\lambda_{i}(x)} 2 + \left( \frac 3 4 - \frac {\lambda_{i}(x)} 2 \right) \frac 4 3 t \right)
& \text{if } t <    \frac 3 4,
            \\
            (x,t)
            & \text{if } t \geq \frac 3 4.
        \end{cases}
    \end{align*}
    Note that $\Xi_{i}$ is an embedding. 
    Its image $( \overline{U_{i}} \times [0,1] ) \setminus O_{i,\mytent}$ is closed in $\overline{U_{i}} \times [0,1]$. 
    Similarly, $Q_{i,\mytent} \cap ( \overline{U_{i}} \times [0,1] )$ is a closed set. 
    See also Figure~\ref{figure:illustrationofxi}. 
    
    \item 
    Following up on this, we introduce the transformations 
    \begin{align*}
        g_{i} : X \to X,
        \quad 
        y \mapsto 
        \begin{cases} 
            c_{i} \circ \Xi_{i} \circ c^{-1}_{i}(y)
            & \text{if } y \in W_{i},
            \\
            y
            & \text{if } y \notin W_{i}.
        \end{cases}
    \end{align*}
    By construction, $g_{i|W_{i}}$ is an embedding of $W_{i}$ into itself.
    
    We know by Theorem~\ref{theorem:yinuowangyoumademesmile} that $c_{i}( \supp(\lambda_{i}) \times [0,0.75] )$ is closed in $X$.
    Hence its complement and the open set $c_{i}( U_{i} \times [0,1)$ are two open sets that cover $X$ and over each of which $g_{i}$
    is a homeomorphism onto its image. It is now clear that $g_{i}$ is an embedding of $X$ into itself $X$.
    
    We understand that $g_{i}$ has closed image in $X$ because its complement in $X$ is the set $c_{i}(O_{i,\mytent})$, which is open in $X$.
    Theorem~\ref{theorem:yinuowangyoumademesmile} also implies that $c_{i}(Q_{i,\mytent}) \subseteq c_{i}( \supp(\lambda_{i}) \times [0,0.75] )$ is closed in $X$.

    \item We outline the recursive construction that we will use. 
    First, we use the collar $c_1$ to pull $B$ into $X$, 
    thus finding a partial definition $h_1$ of the global collar over the hypograph of $\lambda_1$. 
    We also use $g_1$ to deform all the other collars, 
    thus ensuring that the image of $h_1$ is not moved from here on. 
    Next, we use the collar $c_2$ to pull $g_1(B)$ into $g_1(X)$, 
    thus finding a partial definition $h_2$ of the global collar over the hypograph of $\lambda_1+\lambda_2$,
    and again we use $g_2$ to deform all the other collars.
    We keep iterating this procedure, defining partial collars $h_{i}$ 
    and always using modified collars that already incorporate deformations from previous steps.
    See also Figure~\ref{figure:illustrationofcollarconstruction}.
    
    We abbreviate 
    \[
        \lambda_{\Sigma,i}(x) := \lambda_{1}(x) + \lambda_{2}(x) + \cdots + \lambda_{i}(x), \quad i \in \bbN_0, \quad x \in B.
    \]
    We introduce two more auxiliary sets: 
    \begin{align*}
        Q_{\Sigma,i} &:= \left\{ (x,t) \in B \times [0,1] \suchthat* t \leq \sum\nolimits_{j=1}^{i} \lambda_{j}(x) \right\}, 
        \\
        O_{\Sigma,i} &:= \left\{ (x,t) \in B \times [0,1] \suchthat* t <    \sum\nolimits_{j=1}^{i} \lambda_{j}(x) \right\}.
    \end{align*}
    Each $Q_{\Sigma,i}$ is closed and each $O_{\Sigma,i}$ is open in $B \times [0,1]$, respectively. 
    Note that $Q_{\Sigma,0} = B \times \{0\}$. 

    For notational convenience, we also introduce the homeomorphisms
    \begin{align*}
        q_{i}
        : 
        Q_{\Sigma,i} \setminus O_{\Sigma,i-1}
        \to 
        Q_{i,\mytent},
        \quad 
        (x,t) \mapsto \left( x, ( t - \lambda_{\Sigma,i-1}(x) ) / 2 \right)
        .
    \end{align*}
    In what follows, we abbreviate 
    \begin{align*}
        G_{i} := g_1 \circ g_2 \circ \dots \circ g_{i}, \quad i \in \bbN_0.
    \end{align*}
    We now introduce a sequence of mappings. 
    We first define $h_{1} : Q_{\Sigma,1} \to X$ via 
    \begin{align*}
        h_{1}(x,t)
        =
        \begin{cases} 
            c_{1}\left( x, \frac 1 2 t \right) 
            & \text{ if } x \in U_{i}, \; 0 < t \leq \lambda_1(x)
            ,
            \\
            x
            & \text{ otherwise } x \notin U_{i}
            .
        \end{cases}
    \end{align*}
    When $i > 1$ and $h_{i-1} : Q_{\Sigma,i-1} \to X$ has already been defined, 
    then we recursively define the mapping $h_{i} : Q_{\Sigma,i} \to X$ via 
    \begin{align*}
        h_{i}
        (x,t) 
        = 
        \begin{cases} 
            G_{i-1} \circ c_{i}\left( x, \frac{1}{2}( t - \lambda_{\Sigma,i-1}(x) ) \right) 
& \text{ if } x \in U_{i}, \quad \lambda_{\Sigma,i-1}(x) < t \leq \lambda_{\Sigma,i}(x),
\\
            h_{i-1}(x,t) 
            & \text{ otherwise. }
        \end{cases}
    \end{align*}
    Let us study that sequence of mappings.

    \item 
    First, $h_{1}$ is continuous, as evident by the pasting lemma, 
    and $h_{1}( x, \lambda_1(x) ) = g_1(x) = G_{1}(x)$ for all $x \in B$. 
    Next, let $i > 1$.
    If $G_{i-1}(x) = h_{i-1}( x, \lambda_{\Sigma,i-1}(x) )$ for every $x \in B$,
    then 
    \begin{align*}
        G_{i}( x )
        &=
        G_{i-1} \circ g_{i}( x )
=
        G_{i-1} \circ c_{i}( x, \lambda_{i}(x) / 2 )
        =
        h_{i}( x, \lambda_{\Sigma,i}(x) )
    \end{align*}
    for all $x \in B$ again. 
    Going further, $h_i$ is defined via two distinct mappings into $X$,
    themselves defined on the sets $Q_{\Sigma,i-1}$ and $Q_{\Sigma,i} \setminus O_{\Sigma,i-1}$, respectively, 
    which are both closed in $B \times [0,1]$.
    They agree on the intersection of these sets: 
    \begin{align*}
        G_{i-1} \circ c_{i}( x, 0 ) = G_{i-1}( x ) = h_{i-1}( x, \lambda_{\Sigma,i-1}(x) ), 
        \quad 
        x \in B.
    \end{align*}
    If $h_{i-1}$ is continuous, then the pasting lemma implies that $h_{i}$ is continuous. 
    
    \item 
Furthermore, 
    $h_{i}$ is injective over $Q_{\Sigma,j} \setminus O_{\Sigma,j-1}$ when $1 \leq j \leq i$,
    since there it equals $G_{j-1} \circ c_{j} \circ q_{j}$ by definition.
    However, when $j < l \leq i$, then the image of $h_{i}$ over $Q_{\Sigma,l} \setminus O_{\Sigma,l-1}$
    is a subset of the image of $G_{j-1} \circ g_{j}$. 
    The image of $g_{j}$ is disjoint from $c_{j}(O_{j,\mytent})$, and $G_{j-1}$ is an embedding. 
    We conclude that $h_{i}$ is injective. 
    
    \item 
    Next, we show that $h_{i}$ maps open subsets of its domain onto relatively open subsets of its image by a very similar argument. 
    First, we already know that $G_{i-1}^{-1} h_{i}$ equals the embedding $c_{i} q_{i}$ over the set $Q_{\Sigma,i} \setminus O_{\Sigma,i-1}$.
    
    Let now $1 \leq l \leq i-1$ and assume that $G_{l}^{-1} h_{i}$
    restricted to $Q_{\Sigma,i} \setminus O_{\Sigma,l}$ maps relatively open subsets onto relatively open subsets of its image.
    Then the same is true for $g_{l} G_{l}^{-1} h_{i} = G_{l-1}^{-1} h_{i}$.
    We recall that $g_{l}(X)$ is closed in $X$.
    Taking the intersection with $G_{l-1}^{-1} h_{i}\left( Q_{\Sigma,i} \setminus O_{\Sigma,l-1} \right)$,
    we obtain a relatively closed subset of the latter 
    that equals $G_{l-1}^{-1} h_{i}\left( Q_{\Sigma,i} \setminus O_{\Sigma,l} \right)$.
    
    We also know that $c_{l} q_{l}$ is an embedding of $Q_{\Sigma,l} \setminus O_{\Sigma,l-1}$.
    Due to Theorem~\ref{theorem:yinuowangyoumademesmile},
    we know that $c_{l}( Q_{l,\mytent})$ is closed in $X$.
    Its intersection with $G_{l-1}^{-1} h_{i}\left( Q_{\Sigma,i} \setminus O_{\Sigma,l-1} \right)$
    is a relatively closed subset of the latter 
    and equals $G_{l-1}^{-1} h_{i}\left( Q_{\Sigma,l} \setminus O_{\Sigma,l-1} \right)$.
    
    In combination with Lemma~\ref{lemma:joiningrelativeopensets},
    we see that $G_{l-1}^{-1} h_{i}$ restricted to $Q_{\Sigma,i} \setminus O_{\Sigma,l-1}$
    maps relatively open subsets onto relatively open subsets of its image. 
    Iterating the recursive argument, we finally establish that $h_{i}$ maps relatively open subsets of its domain 
    onto relatively open subsets of its image. 

    \item 
    We show that $h_{i}( O_{\Sigma,i} )$ is open in $X$ by a different recursive argument.
    The definition of $h_{i}$ clearly implies that the following set is open in $X$:
    \begin{align*}
        G_{i-1}^{-1} 
        h_{i}\left( O_{\Sigma,i} \setminus O_{\Sigma,i-1} \right) 
        &
        = 
        c_{{i}}q_{i}\left( O_{\Sigma,i} \setminus O_{\Sigma,i-1} \right)
= 
        c_{{i}}\left( O_{{i},\mytent} \right)
        .
    \end{align*}
    Let now $1 \leq l \leq i-1$ and assume that 
    $ G_{l}^{-1} 
        h_{i}\left( O_{\Sigma,i} \setminus O_{\Sigma,l} \right) 
    $ is open in $X$.
    So its complement, which we call $D_l$, is closed in $X$. 
    Therefore $g_{l}(D_l)$ is closed in $g_{l}(X)$ and thus closed in $X$. 
    But the complement of $g_{l}(D_l)$ in $X$ equals 
    $ G_{l-1}^{-1} 
        h_{i}\left( O_{\Sigma,i} \setminus O_{\Sigma,l-1} \right),
    $ which must therefore be open in $X$. 
    The recursion now shows that $h_{i}\left( O_{\Sigma,i} \setminus O_{\Sigma,0} \right) = h_{i}( O_{\Sigma,i} )$ is open in $X$.

    \item So $h_{i}$ is a continuous bijection onto its image which maps open sets onto sets relatively open within its image. 
    We conclude that $h_{i}$ is an embedding.
    Moreover, we have already seen that $h_{i}( O_{\Sigma,{i}} ) = X \setminus G_{i}(X)$,
    which is clearly open in $X$. 
    
    \item Since $\calU$ is a locally finite open cover of $B$,
    for every $x \in B$ we find a relatively open neighborhood $V_x \subseteq B$ and a finite non-empty set $N(x) \subseteq \bbN$ satisfying 
    \begin{align*}
        j \in N(x) \quad\equivalent\quad \overline{U_j} \cap V_x \neq \emptyset \quad\equivalent\quad x \in \overline{U_j}.
    \end{align*}
    We write $i(x) := \max N(x)$ for any $x \in B$. 
    Note that $i(y) \leq i(x)$ for all $y \in V_x$. 
    Moreover, $\lambda_{\Sigma,i(y)}(y) = \lambda_{\Sigma,i(x)}(y) = 1$ for $y \in V_x$.

    \item We show that $h_{m}(x,t) = h_{i(x)}(x,t)$ for $t \in [0,1]$ and all $m > i(x)$ via a recursive argument. 
    Suppose that $m \in \bbN$ is greater than $i(x)$.
    Then $x \notin U_m$ implies that $\lambda_m(x) = 0$ and thus $(x,t) \notin Q_{\Sigma,m} \setminus Q_{\Sigma,m-1}$.
    By definition, $h_{m}(x,t) = h_{m-1}(x,t)$,
    and applying this argument recursively proves $h_{m}(x,t) = h_{i(x)}(x,t)$. 
    
    \item In combination, we define 
    \begin{align*}
        h : B \times [0,1] \to X, \quad (x,t) \mapsto h_{i(x)}(x,t).
    \end{align*}
    We notice that 
$h(x,0) = h_{i(x)}(x,0) = h_{1}(x,0) = x$ for $x \in B$.
It remains to be shown that $h$ is an embedding and that $h( B \times [0,1) )$ is open in $X$. 
    
    \item We first show that $h$ is injective. 
    Let $x, y \in B$ and let $s, t \in [0,1]$ and let $m \subseteq \bbN$ be finite with $i(x), i(y) \leq m$.
    Then 
    \[
        h_{i(x)}(x,s) = h_{m}(x,s) = h(x,s),
        \quad 
        h_{i(y)}(y,t) = h_{m}(y,t) = h(y,t).
    \]
    As shown above, $h_m$ is injective, and we conclude that $h$ is injective.

    \item We show that $h$ is a local homeomorphism onto its image.
    Let $x \in X$ and let the set $V_x$ be as above. 
    We already know that $h_{i(x)}$ is an embedding over the set $V_x \times [0,1]$, which is open in $B \times [0,1]$. 
    Since $h$ equals $h_{i(x)}$ over $V_x \times [0,1]$, we see that $h$ is locally a homeomorphism.
    
    \item 
    So $h$ is injective and thus bijective onto its image. 
    Since $h$ is locally a homeomorphism,
    we conclude that $h$ is a homeomorphism onto its image.
    
    \item We want to show that $h( B \times [0,1) )$ is open in $X$. 
    Again, for any $x \in X$ we let $V_x$ be the set as above. 
Then 
    \[
        h( B \times [0,1) )
        =
        \bigcup_{ x \in B }
        h( V_x \times [0,1) )
        =
        \bigcup_{ x \in B }
        h_{i(x)}( V_x \times [0,1) )
    \] 
    is the union of sets open in $X$, and thus open in $X$ itself. 
    \end{enumerate}
    We have shown that $h$ is a collar of $B$ in $X$.
    This completes the proof. 
\end{proof}

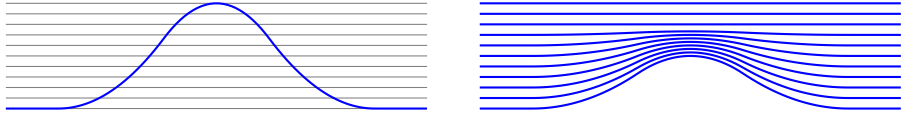
\begin{figure}[t]
    \centering
    \begin{tikzpicture}[scale=1.40]
        \pgfmathdeclarefunction{lambda}{2}{\pgfmathparse{ifthenelse(
                0 <= #1-#2 && #1-#2 < 1, 
                (#1-#2-0)^2 / (3/2),
                ifthenelse(
                    1 <= #1-#2 && #1-#2 <= 2,  
                    ( 2*(1-(#1-#2-1))*(#1-#2-1) + 1 ) / (3/2),    
                    ifthenelse(
                        2 <= #1-#2 && #1-#2 <= 3,  
                        (1-(#1-#2-2))*(1-(#1-#2-2)) / (3/2)  ,    
                        0                        
                    )                        
                )
            )
            }}
        
        \pgfmathdeclarefunction{xi}{2}{\pgfmathparse{ifthenelse(
                #1 <= 0.75, 
                (#2/2) + ( (3/4) - (#2/2) ) * (4/3) * #1
                ,
                #1
            )
            }}
        
        \begin{scope}[xshift=1.5cm]
        \draw[color=gray] (-6.5,0.0) -- (-2.5,0.0);
        \draw[color=gray] (-6.5,0.1) -- (-2.5,0.1);
        \draw[color=gray] (-6.5,0.2) -- (-2.5,0.2);
        \draw[color=gray] (-6.5,0.3) -- (-2.5,0.3);
        \draw[color=gray] (-6.5,0.4) -- (-2.5,0.4);
        \draw[color=gray] (-6.5,0.5) -- (-2.5,0.5);
        \draw[color=gray] (-6.5,0.6) -- (-2.5,0.6);
        \draw[color=gray] (-6.5,0.7) -- (-2.5,0.7);
        \draw[color=gray] (-6.5,0.8) -- (-2.5,0.8);
        \draw[color=gray] (-6.5,0.9) -- (-2.5,0.9);
        \draw[color=gray] (-6.5,1.0) -- (-2.5,1.0);
        \draw[  blue, opaque, thick, variable=\x,  domain=-6.5:-2.5, samples=100] plot (\x, {lambda(\x,-6.0)});
        \end{scope}
        
        \draw[  blue, thick, variable=\x,  domain=-0.5:3.5, samples=100] plot (\x, { xi( 1.0, lambda(\x,0) ) });
        \draw[  blue, thick, variable=\x,  domain=-0.5:3.5, samples=100] plot (\x, { xi( 0.9, lambda(\x,0) ) });
        \draw[  blue, thick, variable=\x,  domain=-0.5:3.5, samples=100] plot (\x, { xi( 0.8, lambda(\x,0) ) });
        \draw[  blue, thick, variable=\x,  domain=-0.5:3.5, samples=100] plot (\x, { xi( 0.7, lambda(\x,0) ) });
        \draw[  blue, thick, variable=\x,  domain=-0.5:3.5, samples=100] plot (\x, { xi( 0.6, lambda(\x,0) ) });
        \draw[  blue, thick, variable=\x,  domain=-0.5:3.5, samples=100] plot (\x, { xi( 0.5, lambda(\x,0) ) });
        \draw[  blue, thick, variable=\x,  domain=-0.5:3.5, samples=100] plot (\x, { xi( 0.4, lambda(\x,0) ) });
        \draw[  blue, thick, variable=\x,  domain=-0.5:3.5, samples=100] plot (\x, { xi( 0.3, lambda(\x,0) ) });
        \draw[  blue, thick, variable=\x,  domain=-0.5:3.5, samples=100] plot (\x, { xi( 0.2, lambda(\x,0) ) });
        \draw[  blue, thick, variable=\x,  domain=-0.5:3.5, samples=100] plot (\x, { xi( 0.1, lambda(\x,0) ) });
        \draw[  blue, thick, variable=\x,  domain=-0.5:3.5, samples=100] plot (\x, { xi( 0.0, lambda(\x,0) ) });            
    \end{tikzpicture}         
    \caption{Illustration of a partition of unity function associated with a local collar (left)
    and the mapping $\Xi$ associated with that function (right),
    as used in the proof of Theorem~\ref{theorem:collaring:locallyfinitecountable}.}\label{figure:illustrationofxi}
\end{figure}

\begin{figure}[t]
    \centering
    \begin{tikzpicture}[scale=1.40]
        \pgfmathdeclarefunction{lambda}{2}{\pgfmathparse{ifthenelse(
                0 <= #1-#2 && #1-#2 < 1, 
                (#1-#2-0)^2 / (3/2),
                ifthenelse(
                    1 <= #1-#2 && #1-#2 <= 2,  
                    ( 2*(1-(#1-#2-1))*(#1-#2-1) + 1 ) / (3/2),    
                    ifthenelse(
                        2 <= #1-#2 && #1-#2 <= 3,  
                        (1-(#1-#2-2))*(1-(#1-#2-2)) / (3/2)  ,    
                        0                        
                    )                        
                )
            )
            }}
        
        \draw ( 5.75,0.7) node {$\color{white}X$};
        \draw (-0.25,0.0) node {$B$};
        \draw ( 5.75,0.0) node {$0$};
        \draw ( 5.75,1.0) node {$1$};
        
        \draw[very thin, color=gray] (-0.0,0.0) -- (+5.5,0.0);
        \draw[very thin, color=gray] (-0.0,0.1) -- (+5.5,0.1);
        \draw[very thin, color=gray] (-0.0,0.2) -- (+5.5,0.2);
        \draw[very thin, color=gray] (-0.0,0.3) -- (+5.5,0.3);
        \draw[very thin, color=gray] (-0.0,0.4) -- (+5.5,0.4);
        \draw[very thin, color=gray] (-0.0,0.5) -- (+5.5,0.5);
        \draw[very thin, color=gray] (-0.0,0.6) -- (+5.5,0.6);
        \draw[very thin, color=gray] (-0.0,0.7) -- (+5.5,0.7);
        \draw[very thin, color=gray] (-0.0,0.8) -- (+5.5,0.8);
        \draw[very thin, color=gray] (-0.0,0.9) -- (+5.5,0.9);
        \draw[very thin, color=gray] (-0.0,1.0) -- (+5.5,1.0);
        \draw[   red, opaque, thick, variable=\x,  domain=-0.0:+5.5, samples=100] plot (\x, {lambda(\x, 2.0)});
        \draw[  blue, opaque, thick, variable=\x,  domain=-0.0:+5.5, samples=100] plot (\x, {lambda(\x, 0.5)});
    \end{tikzpicture} 
    ${}$\\
    \vspace{0.5cm}
    \centering
    \begin{tikzpicture}[scale=1.40]
        \pgfmathdeclarefunction{lambda}{2}{\pgfmathparse{ifthenelse(
                0 <= #1-#2 && #1-#2 < 1, 
                (#1-#2-0)^2 / (3/2),
                ifthenelse(
                    1 <= #1-#2 && #1-#2 <= 2,  
                    ( 2*(1-(#1-#2-1))*(#1-#2-1) + 1 ) / (3/2),    
                    ifthenelse(
                        2 <= #1-#2 && #1-#2 <= 3,  
                        (1-(#1-#2-2))*(1-(#1-#2-2)) / (3/2)  ,    
                        0                        
                    )                        
                )
            )
            }}
        
        \pgfmathdeclarefunction{xi}{2}{\pgfmathparse{ifthenelse(
                #1 <= 0.75, 
                (#2/2) + ( (3/4) - (#2/2) ) * (4/3) * #1
                ,
                #1
            )
            }}
        
\draw ( 5.75,0.0) node {$\color{white}0$};
        \draw ( 5.75,1.0) node {$\color{white}1$};
        \draw (-0.25,0.0) node {$B$};
        \draw ( 5.75,0.7) node {$X$};
        
        \draw[  gray, very thin, dashed, variable=\x, domain=0.0:5.5, samples=  5] plot (\x, 1.3 );
        \draw[  gray, very thin, variable=\x, domain=0.0:5.5, samples=  5] plot (\x, 1.2 );
        \draw[  gray, very thin, variable=\x, domain=0.0:5.5, samples=  5] plot (\x, 1.1 );
        \draw[  gray, very thin, variable=\x, domain=0.0:5.5, samples=  5] plot (\x, { 1.0 });
        \draw[  gray, very thin, variable=\x, domain=0.0:5.5, samples=  5] plot (\x, { 0.9 });
        \draw[  gray, very thin, variable=\x, domain=0.0:5.5, samples=  5] plot (\x, { 0.8 });
        \draw[  gray, very thin, variable=\x, domain=0.0:5.5, samples=100] plot (\x, { xi( 0.7, lambda(\x,0.5) ) });
        \draw[  gray, very thin, variable=\x, domain=0.0:5.5, samples=100] plot (\x, { xi( 0.6, lambda(\x,0.5) ) });
        \draw[  gray, very thin, variable=\x, domain=0.0:5.5, samples=100] plot (\x, { xi( 0.5, lambda(\x,0.5) ) });
        \draw[  gray, very thin, variable=\x, domain=0.0:5.5, samples=100] plot (\x, { xi( 0.4, lambda(\x,0.5) ) });
        \draw[  gray, very thin, variable=\x, domain=0.0:5.5, samples=100] plot (\x, { xi( 0.3, lambda(\x,0.5) ) });
        \draw[  gray, very thin, variable=\x, domain=0.0:5.5, samples=100] plot (\x, { xi( 0.2, lambda(\x,0.5) ) });
        \draw[  gray, very thin, variable=\x, domain=0.0:5.5, samples=100] plot (\x, { xi( 0.1, lambda(\x,0.5) ) });

\draw[   red, thick, variable=\x, domain=0.0:5.5, samples=100] plot (\x, { xi( lambda(\x,2.0)/2, lambda(\x,0.5) ) });            
        \draw[  blue, thick, variable=\x, domain=0.0:5.5, samples=100] plot (\x, { xi( 0.0,              lambda(\x,0.5) ) });            
\end{tikzpicture} 
    ${}$\\
    \caption{Illustration of the collar construction for Theorem~\ref{theorem:collaring:locallyfinitecountable}.
    We have two partition of unity functions within the parameter space $B \times [0,1]$ (top).
    As push $B$ into the space $X$ in accordance to the first collar, the other collars are pushed along that deformation as well (bottom).
    The trajectory of the points $B$ along all those pushes determines the final collar.}\label{figure:illustrationofcollarconstruction}
\end{figure}
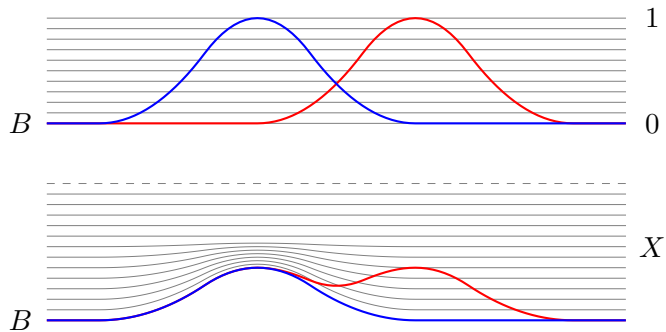

We are now in a position to easily prove the main outcome of this section. 

\begin{theorem}\label{theorem:collaring:paracompactinambient}
    Let $B$ be a closed subset of a Hausdorff space $X$.
    If $B$ is locally collared and paracompact in $X$, then $B$ is collared.
\end{theorem}
\begin{proof}
    Since $B$ is locally collared in $X$, there exists a relatively open cover $\calU = ( U_{\alpha} )_{\alpha \in \kappa}$ of $B$ and a collection of local collars 
    \[
        c_\alpha : \overline{U_\alpha} \times [0,1] \to X, \quad \alpha \in \kappa. 
    \]
    Since $B$ is paracompact in $X$,
    there exists a locally finite cover $\calU = \bigcup_{i\in\bbN} \calU_{i}$,
    where $\calU_{i} = ( U_{i,\alpha} )_{\alpha \in \kappa}$ is a relatively open family such that
    for every $i \in \bbN$ and $\alpha \in \kappa$ we have the inclusion $U_{i,\alpha} \subseteq U_{\alpha}$. 
We set $U_{i} := \cup_{\alpha \in \kappa} U_{i,\alpha}$.
    We also notice that $\overline{ U_{i} } = \cup_{\alpha\in\kappa} \overline{U_{i,\alpha}}$. 
    By Lemma~\ref{lemma:collaringdiscrete}, there exist local collars 
    \begin{align*}
        c_{i} : \overline{U_{i}} \times [0,1] \to X.
    \end{align*}
    Moreover, the family $(U_{i})_{i \in \bbN}$ is locally finite and countable. 
    Using Theorem~\ref{theorem:collaring:locallyfinitecountable} and that $B$ is normal, 
    we conclude that $B$ is collared in $X$.
\end{proof}

\begin{corollary}\label{corollary:collaringwithinparacompactspaces}
    If $X$ is a paracompact Hausdorff space and $B \subseteq X$ is a locally collared closed subspace, then $B$ is collared in $X$.
\end{corollary}

\begin{remark}
    Brown's original proof relies on Michael's theorem 
    and verifies its three conditions: 
    (i) subsets of collared sets are collared 
    (ii) disjoint unions of collared sets are collared, and 
    (iii) finite unions of collared sets are collared. 
    The first condition is obviously true. 
    Our Lemma~\ref{lemma:collaringdiscrete} proves a weaker version of the second condition, 
    which says that discrete unions of collared sets are collared,
    whereas Theorem~\ref{theorem:collaring:locallyfinitecountable} proves a stronger version of the third condition,
    which says that countable locally finite unions of collared sets are collared. 
\end{remark}

\begin{remark}\label{remark:stronglyparacompactsituation}
    In Connelly's contribution~\cite{connelly1971new} and subsequent works~\cite{gauld2014non,baillif2022collared},
    one glues the space $B \times [-1,0]$ along $B \subseteq X$ and constructs the collar 
    by locally pulling $X$ \emph{down} into that extended space. 
By contrast, our approach is visualized best as pushing the collared set $B$ up into $X$,
    locally and step by step. Each push deforms not only the local portion of $B$ (or rather its image after the preceding pushes), but also to all surrounding local collars. 
    The trajectory of each $x \in B$ finalizes after a finite number of operations and defines the global collar. 
\end{remark}

\section{The strongly paracompact case}\label{section:stronglyparacompact}
We address an interesting special case: 
we prove the collar theorem for locally collared sets that are strongly paracompact within their Hausdorff ambient space. 
Proofs for this special case can be found in the monograph by Gauld~\cite[Section~3]{gauld2014non} and a contribution by Baillif~\cite{baillif2022collared}.
Even though this is a special case of the more general result in the previous section, 
this manuscript provides a separate proof to complete the exposition and to contribute a new perspective on a known theorem. 

The central statement in this section is the following generalization of Theorem~\ref{theorem:collaring:locallyfinitecountable};
more context and discussion will follow later.

\begin{theorem}\label{theorem:collaring:locallyfinitegeneralized}
    Let $B$ be a normal subset of some Hausdorff space $X$.
    Let 
    \begin{align*}
        \calC = \left( c_{\alpha} : \overline{U_{\alpha}} \times [0,1] \to X \right)_{\alpha \in \kappa} 
    \end{align*}
    be a collar cover, indexed over a set $\kappa$, 
    such that $\calU = \left\{ U_\alpha \suchthat* \alpha \in \kappa \right\}$ is a locally finite open cover of $B$.
    
    Suppose $\kappa$ is well-ordered and that every $x \in B$ has a relatively open neighborhood $S_{x} \subseteq B$
    such that for each $\alpha \in \kappa$ there are only finitely-many $\beta \in \kappa$ with $\beta < \alpha$ and 
    \begin{align*}
        c_{\alpha}( (S_{x} \cap \overline{U_{\alpha}}) \times [0,1] )
        \cap 
        c_{\beta} ( \overline{U_{\beta}} \times [0,1] )
        \neq 
        \emptyset
        .
    \end{align*}
    Then $B$ is collared.
\end{theorem}
\begin{proof} 
    We organize the proof into several steps. 
    \begin{enumerate}[1., wide=10pt, itemindent=\parindent, leftmargin=0pt, topsep=0pt, itemsep=0pt]
    \item 
    Since $B$ is normal and $\calU$ is locally finite, 
    there exists a continuous partition of unity $( \lambda_{\alpha} : B \rightarrow \bbR )_{\alpha \in \kappa}$ subordinate to the cover $\calU$. 
    Each $\lambda_{\alpha} : B \to [0,1]$ is non-zero over an open set $V_\alpha$ satisfying $\overline{V_{\alpha}} \subseteq U_{\alpha}$.
    Notice that $\calV = \left\{ V_\alpha \suchthat* \alpha \in \kappa \right\}$ is a locally finite open cover of $B$.
    
    \item 
    Given any subset $V \subseteq B$ and $\alpha \in \kappa$, 
    we define the two auxiliary sets 
    \begin{align*}
Q_{V,\alpha,\mytent} &:= \left\{ (x,t) \in V \times [0,1] \suchthat* t \leq \frac 1 2 \lambda_{\alpha}(x) \right\}, 
        \\
        O_{V,\alpha,\mytent} &:= \left\{ (x,t) \in V \times [0,1] \suchthat* t <    \frac 1 2 \lambda_{\alpha}(x) \right\}.
    \end{align*}
    Note that $Q_{V,\alpha,\mytent}$ is closed and $O_{V,\alpha,\mytent}$ is open in $V \times [0,1]$, respectively.

    \item For each $\alpha \in \kappa$, we introduce the homeomorphism 
    \begin{align*} 
        \Xi_{\alpha} : \overline{U_{\alpha}} \times [0,1] \to \overline{U_{\alpha}} \times [0,1]
    \end{align*}
    defined by 
    \begin{align*} 
        \Xi_{\alpha}(x,t)
        = 
        \begin{cases} 
            \left( x, \frac {\lambda_{\alpha}(x)} 2 + \left( \frac 3 4 - \frac {\lambda_{\alpha}(x)} 2 \right) \frac 4 3 t \right)
            & \text{if } t <    \frac 3 4,
            \\
            (x,t)
            & \text{if } t \geq \frac 3 4.
        \end{cases}
    \end{align*}
    Note that $\Xi_{\alpha}$ is an embedding. 
    Its image $( \overline{U_{\alpha}} \times [0,1] ) \setminus O_{\overline{U_{\alpha}},\alpha,\mytent}$ is closed in $\overline{U_{\alpha}} \times [0,1]$.

    \item 
    Following up on this, we introduce the transformations 
    \begin{align*}
        g_{\alpha} : X \to X,
        \quad 
        y \mapsto 
        \begin{cases} 
            c_{\alpha} \circ \Xi_{\alpha} \circ c^{-1}_{\alpha}(y)
            & \text{if } y \in W_{\alpha},
            \\
            y
            & \text{if } y \notin W_{\alpha}.
        \end{cases}
    \end{align*}
    By construction, $g_{\alpha|W_{\alpha}}$ is an embedding of $W_{\alpha}$ into itself.
    
    We know by Theorem~\ref{theorem:yinuowangyoumademesmile} that $c_{\alpha}( \supp(\lambda_{\alpha}) \times [0,0.75] )$ is closed in $X$.
    Hence its complement and the open set $c_{\alpha}( U_{\alpha} \times [0,1)$ are two open sets that cover $X$ and over each of which $g_{\alpha}$
    is a homeomorphism onto its image. It is now clear that $g_{\alpha}$ is an embedding of $X$ into itself $X$.
    
    We understand that $g_{\alpha}$ has closed image in $X$
    because its complement in $X$ is the set $c_{\alpha}(O_{\overline{U_{\alpha}},\alpha,\mytent})$,
    which is open in $X$.

    \item Since $\calU$ is locally finite, Lemma~\ref{lemma:closurecondition} is applicable:
    for every $x \in B$ there exists a finite maximal set $\kappa(x) \subseteq \kappa$
    such that every open neighborhood of $x$ intersects $V_{\alpha}$ for all $\alpha \in \kappa(x)$. 
    By definition, $x \in \overline{V_{\alpha}}$ for each $\alpha \in \kappa(x)$.
    We conclude that $x \in U_{\alpha}$ for each $\alpha \in \kappa(x)$.     
    We fix an open neighborhood of $x$: 
    \[
        V_{x,\kappa} 
        := 
\bigcap\left\{ X \setminus \overline{V_{\beta}} \suchthat* \beta \in \kappa \setminus \kappa(x) \right\}
        .
    \] 
    By construction, $V_{x,\kappa}$ intersects $V_{\beta} \in \calV$ if and only if $\beta \in \kappa(x)$. 
    In particular, $\kappa(y) \subseteq \kappa(x)$ for all $y \in V_{x,\kappa}$. 
    
    Now let $S_{x}$ be the neighborhood as in the statement of the theorem.
    We set 
    \[
        V_{x} := S_{x} \cap V_{x,\kappa}. 
    \]
    We define $\kappa'(x) \subseteq \kappa$ as follows:
    $\beta \in \kappa'(x)$ if and only if there exists $\alpha \in \kappa(x)$
    with $\beta \leq \alpha$ and 
    \begin{align*}
        c_{\alpha}( (V_{x} \cap U_{\alpha}) \times [0,1] )
        \cap 
        c_{\beta} ( U_{\beta} \times [0,1] )
        \neq 
        \emptyset
        .
    \end{align*}
    By assumption, $\kappa'(x)$ is finite. Note that $\kappa'(y) \subseteq \kappa'(x)$ for all $y \in V_{x}$. 
    
    We study a family of local mappings that we eventually paste together to get the global collar.
    For the time being, we assume that $V \subseteq B$ be an open set and that $\mu \subseteq \kappa$ is a finite set of indices 
    $\alpha_1 < \dots < \alpha_N$ such that $\kappa'(y) \subseteq \mu$ for each $y \in V$.
    A possible choice, to be used later, is $V = V_x$ and $\mu = \kappa'(x)$ for some arbitrary $x \in B$, in which case $N = |\kappa'(x)|$.
    
    We abbreviate 
    \[
        \lambda_{V,\mu,\Sigma,i}(x) := \lambda_{\alpha_1}(x) + \lambda_{\alpha_2}(x) + \cdots + \lambda_{\alpha_i}(x), \quad 0 \leq i \leq N, \quad x \in V.
    \]
    We introduce two more auxiliary sets: 
    \begin{align*}
        Q_{V,\mu,\Sigma,i} &:= \left\{ (x,t) \in V \times [0,1] \suchthat* t \leq \sum\nolimits_{j=1}^{i} \lambda_{\alpha_j}(x) \right\},
        \\
        O_{V,\mu,\Sigma,i} &:= \left\{ (x,t) \in V \times [0,1] \suchthat* t    < \sum\nolimits_{j=1}^{i} \lambda_{\alpha_j}(x) \right\}.
    \end{align*}
    Each $Q_{V,\mu,\Sigma,i}$ is closed and each $O_{V,\mu,\Sigma,i}$ is open in $V \times [0,1]$, respectively. 
    Note that $Q_{V,\mu,\Sigma,0} = V \times \{0\}$. 
    
    For notational convenience, we also introduce the homeomorphisms
    \begin{align*}
        q_{V,\mu,i}
        : 
        Q_{V,\mu,\Sigma,i} \setminus O_{V,\mu,\Sigma,i-1}
        \to 
        Q_{V,\alpha_i,\mytent},
        \quad 
        (x,t) \mapsto \left( x, ( t - \lambda_{V,\mu,\Sigma,i-1}(x) ) / 2 \right)
        .
    \end{align*}
    In what follows, we write 
    \begin{align*}
        G_{V,\mu,i} := g_{\alpha_{1}} \circ g_{\alpha_{2}} \circ \cdots \circ g_{\alpha_{i}}.
    \end{align*}
    We define the mapping $h_{V,\mu,1} : Q_{V,\mu,\Sigma,1} \to X$ via 
    \begin{align*}
        h_{V,\mu,1}(x,t)
        = 
        \begin{cases} 
            c_{\alpha_1}(x,t/2) 
            & \text{ if } (x,t) \in Q_{V,\mu,\Sigma,1  } \setminus Q_{V,\mu,\Sigma,0}
            ,
            \\
            (x,t) 
            & \text{ if } (x,t) \in Q_{V,\mu,\Sigma,0}
            .
        \end{cases}
    \end{align*}
    For $i > 1$, we recursively define the mapping $h_{V,\mu,i} : Q_{V,\mu,i} \to X$ by 
    {\small
    \begin{align*}
        h_{V,\mu,i}
        (x,t) 
        = 
        \begin{cases} 
            h_{V,\mu,i-1}(x,t) 
            & \text{ if } (x,t) \in Q_{V,\mu,\Sigma,i-1},
            \\
            G_{V,\mu,i-1} \circ c_{\alpha_{i}} \circ q_{V,\mu,i}(x,t) 
            & \text{ if } (x,t) \in Q_{V,\mu,\Sigma,i  } \setminus Q_{V,\mu,\Sigma,i-1}.
        \end{cases}
    \end{align*}
    }

    \item 
    First, $h_{V,\mu,1}$ is continuous, due to the pasting lemma, 
    and \[h_{V,\mu,1}( x, \lambda_{\alpha_1}(x) ) = g_{\alpha_1}(x) = G_{V,\mu,1}(x)\] for all $x \in B$. 
    Next, let $i > 1$.
    If \[G_{V,\mu,i-1}(x) = h_{V,\mu,i-1}( x, \lambda_{V,\mu,\Sigma,i-1}(x) )\] for every $x \in V$,
    then 
    \begin{align*}
        G_{V,\mu,i}( x )
        &=
        G_{V,\mu,i-1} \circ g_{\alpha_i}( x )
        \\&=
        G_{V,\mu,i-1} \circ c_{\alpha_i}( x, \lambda_{\alpha_i}(x) / 2 )
        =
        h_{V,\mu,i}( x, \lambda_{V,\mu,\Sigma,i}(x) )
    \end{align*}
    for all $x \in V$ again. 
    Going further, we notice that $h_i$ is defined via two distinct mappings into $X$,
    themselves defined on the sets $Q_{V,\mu,\Sigma,i-1}$ and $Q_{V,\mu,\Sigma,i} \setminus O_{V,\mu,\Sigma,i-1}$, respectively, 
    which are both closed in $V \times [0,1]$.
    They agree on the intersection of these sets: 
    \begin{align*}
        G_{V,\mu,i-1} \circ c_{\alpha_i}( x, 0 ) = G_{V,\mu,i-1}( x ) = h_{V,\mu,i-1}( x, \lambda_{V,\mu,\Sigma,i-1}(x) ), 
        \quad 
        x \in V.
    \end{align*}
    If $h_{V,\mu,i-1}$ is continuous, then the pasting lemma implies that $h_{V,\mu,i}$ is continuous. 
    
    \item 
Furthermore, 
    $h_{V,\mu,i}$ is injective over $Q_{V,\mu,\Sigma,j} \setminus O_{V,\mu,\Sigma,j-1}$ when $1 \leq j \leq i$,
    since there it equals $G_{V,\mu,j-1} \circ c_{\alpha_j} \circ q_{V,\mu,j}$ by definition.
    However, when $j < l \leq i$, then the image of $h_{V,\mu,i}$ over $Q_{V,\mu,\Sigma,l} \setminus O_{V,\mu,\Sigma,l-1}$
    is a subset of the image of $G_{V,\mu,j-1} \circ g_{\alpha_j}$. 
    The image of $g_{\alpha_j}$ is disjoint from $c_{\alpha_j}(O_{V,\alpha_j,\mytent})$.
    As $G_{V,\mu,j-1}$ is an embedding, 
    we conclude that $h_{V,\mu,i}$ is injective.

    \item 
    Next, we show that $h_{V,\mu,i}$ maps open subsets of its domain onto relatively open subsets of its image. 
    First, we already know that $G_{V,\mu,i-1}^{-1} h_{V,\mu,i}$ equals the embedding $c_{\alpha_{i}} q_{V,\mu,i}$ over the set $Q_{V,\mu,\Sigma,i} \setminus O_{V,\mu,\Sigma,l}$.
    
    Let now $1 \leq l \leq i-1$ and assume that $G_{V,\mu,l}^{-1} h_{V,\mu,i}$
    restricted to $Q_{V,\mu,\Sigma,i} \setminus O_{V,\mu,\Sigma,l}$ maps relatively open subsets onto relatively open subsets of its image.
    Then the same is true for $g_{\alpha_{l}} G_{V,\mu,l}^{-1} h_{V,\mu,i}$.
    We also know that $c_{\alpha_{l}} q_{V,\mu,l}$ restricted to $Q_{V,\mu,\Sigma,l} \setminus O_{V,\mu,\Sigma,l-1}$
    maps relatively open subsets to relatively open subsets of its image.
    
    On the one hand, we recall that $g_{\alpha_{l}}(X)$ is closed in $X$.
    Intersecting $g_{\alpha_{l}}(X)$ and $G_{V,\mu,l-1}^{-1} h_{V,\mu,i}\left( Q_{V,\mu,\Sigma,i} \setminus O_{V,\mu,\Sigma,l-1} \right)$,
    we obtain a relatively closed subset of the latter 
    that equals $G_{V,\mu,l-1}^{-1} h_{V,\mu,i}\left( Q_{V,\mu,\Sigma,i} \setminus O_{V,\mu,\Sigma,l-1} \right)$.
    
    On the other hand, 
    due to Theorem~\ref{theorem:yinuowangyoumademesmile},
    we know that $c_{\alpha_{l}}( Q_{\overline{V_{\alpha_{l}}},\alpha_{l},\mytent} )$ is closed in $X$,
    and so is $c_{\alpha_{l}}( Q_{\overline{V_{\alpha_{l}}},\alpha_{l},\mytent} ) \cup B$,
    The intersection of $c_{\alpha_{l}}( Q_{\overline{V_{\alpha_{l}}},\alpha_{l},\mytent} ) \cup B$ with $G_{V,\mu,l-1}^{-1} h_{V,\mu,i}\left( Q_{V,\mu,\Sigma,i} \setminus O_{V,\mu,\Sigma,l-1} \right)$
    is a relatively closed subset of the latter 
    and equals $G_{V,\mu,l-1}^{-1} h_{V,\mu,i}\left( Q_{V,\mu,\Sigma,l} \setminus O_{V,\mu,\Sigma,l-1} \right)$.
    
    We notice that $g_{\alpha_{l}}(\overline{V_{\alpha_{l}}}) = g_{\alpha_{l}}(X) \cap c_{\alpha_{l}} ( Q_{ \overline{V_{\alpha_{l}}},{l},\mytent} )$ and observe 
    \begin{align*}
        &
        G_{V,\mu,l-1}^{-1} h_{V,\mu,i}\left( Q_{V,\mu,\Sigma,i} \setminus O_{V,\mu,\Sigma,l} \right)
        \cap 
        ( g_{\alpha_{l}}(\overline{V_{\alpha_{l}}}) \cup B )
        \\&\qquad 
        =
        G_{V,\mu,l-1}^{-1} h_{V,\mu,i}\left( Q_{V,\mu,\Sigma,l} \setminus O_{V,\mu,\Sigma,l-1} \right)
        \cap 
        ( g_{\alpha_{l}}(\overline{V_{\alpha_{l}}}) \cup B )
        .
    \end{align*}
    Together with Lemma~\ref{lemma:joiningrelativeopensets},
    we see that $G_{V,\mu,l-1}^{-1} h_{V,\mu,i}$ restricted to $Q_{V,\mu,\Sigma,i} \setminus O_{V,\mu,\Sigma,l}$
    maps relatively open subsets onto relatively open subsets of its image. 
    Iterating the recursive argument establishes that $h_{V,\mu,i}$ maps relatively open subsets of its domain 
    onto relatively open subsets of its image.

    \item 
    We show that $h_{V,\mu,i}( O_{V,\mu,\Sigma,i} )$ is open in $X$ by a similar recursive argument.
    The definition of $h_{V,\mu,i}$ clearly implies that the following set is open in $X$:
    \begin{align*}
        &
        G_{V,\mu,i-1}^{-1} 
        h_{V,\mu,i}\left( O_{V,\mu,\Sigma,i} \setminus O_{V,\mu,\Sigma,i-1} \right) 
        \\&\qquad
        = 
        c_{\alpha_{i}} q_{V,\mu,i}\left( O_{V,\mu,\Sigma,i} \setminus O_{V,\mu,\Sigma,i-1} \right)
= 
        c_{\alpha_{i}}\left( O_{ U_{\alpha_{i}} \cap V,{i},\mytent} \right)
        .
    \end{align*}
    Let now $1 \leq l \leq i-1$ and assume that 
    \begin{align*}
        X_{l}
        :=
        G_{V,\mu,l}^{-1} 
        h_{V,\mu,i}\left( O_{V,\mu,\Sigma,i} \setminus O_{V,\mu,\Sigma,l} \right) 
    \end{align*}
    is open in $X$. 
    On the one hand, $g_{\alpha_{l}}(X_l)$ is open in $g_{\alpha_{l}}(X)$, which is closed in $X$.
    On the other hand, 
    $c_{\alpha_{l}} ( Q_{ V \cap \overline{V_{\alpha_{l}}},{l},\mytent} )$ 
    is open in $c_{\alpha_{l}} ( Q_{ \overline{V_{\alpha_{l}}},{l},\mytent} )$, which is closed in $X$.
    
    We notice that $g_{\alpha_{l}}(\overline{V_{\alpha_{l}}}) = g_{\alpha_{l}}(X) \cap c_{\alpha_{l}} \left( Q_{ \overline{V_{\alpha_{l}}},{l},\mytent} \right)$ and observe 
    \begin{align*}
        c_{\alpha_{l}} \left( Q_{ V \cap \overline{V_{\alpha_{l}}},{l},\mytent} \right)
        \cap 
        g_{\alpha_{l}}(\overline{V_{\alpha_{l}}})
        =
        X_{l} 
        \cap 
        g_{\alpha_{l}}(\overline{V_{\alpha_{l}}}).
    \end{align*}
    We employ Lemma~\ref{lemma:joiningrelativeopensets} to conclude that 
    $G_{V,\mu,l-1}^{-1} h_{V,\mu,i}\left( O_{V,\mu,\Sigma,i} \setminus O_{V,\mu,\Sigma,l-1} \right)$ is open in $X$.
    By recursion, $h_{V,\mu,i}\left( O_{V,\mu,\Sigma,i} \setminus O_{V,\mu,\Sigma,0} \right) = h_{V,\mu,i}( O_{V,\mu,\Sigma,i} )$ is open in $X$.

    \item In summary, 
    $h_{V,\mu,i}$ is a continuous bijection onto its image which maps open sets onto sets relatively open within its image. 
    We conclude that $h_{V,\mu,i}$ is a homeomorphism onto its image.
    Moreover, $h_{V,\mu,i}\left( O_{V,\mu,\Sigma,i} \right)$ is open in $X$. 
    Recall that $N = |\mu|$. We see that 
    \[
        Q_{V,\mu,\Sigma,N} = V \times [0,1], \quad O_{V,\mu,\Sigma,N} = V \times [0,1).
    \]
    From here on, $h_{V,\mu} := h_{V,\mu,|\mu|}$. 
    
    \item Given $x \in V$ and $t \in [0,1]$, we verify that $h_{V,\mu}(x,t)$ does not depend on the choice of $\mu$ or the choice of $V$.
    Suppose that $V' \subseteq B$ is open and that $\mu' \subseteq \kappa$ is finite with $\kappa'(y) \subseteq \mu'$,
    and suppose that $V \subseteq V'$ and $\mu \subseteq \mu'$. 
    If $\beta \in \mu' \setminus \mu$, then $\lambda_{\beta}$ vanishes over $V$, 
    and $g_{\beta}$ acts as the identity on any $c_{\alpha}(V \cap U_{\alpha})$.
    Consequently, $h_{V,\mu,|\mu|}(x,t) = h_{V',\mu',|\mu'|}(x,t)$.
    We conclude that 
    \[
        h_{V,\mu}(x,t) = h_{V',\mu'}(x,t)
        .
    \]

    \item In combination, we define 
    \begin{align*}
        h : B \times [0,1] \to X, \quad (x,t) \mapsto h_{V_x,\kappa'(x)}(x,t).
    \end{align*}
    We notice that $h(x,0) = h_{V_x,\kappa'(x)}(x,0) = h_{V_x,\kappa'(x),1}(x,0) = x$ for $x \in B$. 
    It remains to be shown that $h$ is an embedding and that $h( B \times [0,1) )$ is open in $X$. 
    
    \item 
    We show that $h$ is injective. 
    Let $x, y \in B$ and $s, t \in [0,1]$, and let $\mu \subseteq \kappa$ be finite with $\kappa'(x) \cup \kappa'(y) \subseteq \mu$.
    As shown above, $h_{V_x \cup V_y,\mu}$ is injective. Now,
    \[
        h(x,s) = h_{V_x,\kappa'(x)}(x,s) = h_{V_x \cup V_y,\mu}(x,s)
        ,
        \quad 
        h(y,t) = h_{V_y,\kappa'(y)}(y,t) = h_{V_x \cup V_y,\mu}(y,t)
        ,
    \]
    and we conclude that $h$ is injective.

    \item We show that $h$ is a local homeomorphism onto its image.
    $( V_x \times [0,1] )_{x \in B}$ is an open cover of $B \times [0,1]$. 
    For any $x \in X$ and for all $(y,t) \in V_x \times [0,1]$ we establish 
    \begin{align*}
        h(y,t) 
        =
        h_{V_y,\kappa'(y)}(y,t) 
        = 
        h_{V_y \cap V_x,\kappa'(y)}(y,t) 
        = 
        h_{V_x,\kappa'(x)}(y,t)
        .
    \end{align*}
    Hence $h$ equals $h_{V_x,\kappa'(x)}$ over $V_x \times [0,1]$.
    Recall that $h_{V_x,\kappa'(x)}$ over $V_x \times [0,1]$ is an embedding.
    We conclude that $h$ is a local homeomorphism. 

    \item 
    So $h$ is injective and thus bijective onto its image. 
    Since $h$ is locally a homeomorphism, 
    we conclude that $h$ is a homeomorphism onto its image.

    \item We want to show that $h( B \times [0,1) )$ is open in $X$. 
    Obviously,
    \begin{align*}
        h\left( B \times [0,1) \right)
        =
        \bigcup_{x \in B}
        h\left( V_x \times [0,1) \right)
        =
        \bigcup_{x \in B}
        h_{V_x,\kappa'(x)}\left( V_x \times [0,1) \right)
    \end{align*}
    is open, being the union of open sets. 
    \end{enumerate}
    To summarize, $h$ is a collar of $B$ in $X$.
    This completes the proof. 
\end{proof}

\begin{theorem}\label{theorem:collaring:stronglyparacompact}
    Let $B \subseteq X$ be a closed set that is strongly paracompact in a Hausdorff space $X$.
    If $B$ is locally collared, then $B$ is collared.
\end{theorem}
\begin{proof}
    The set $B$ has a relatively open cover $\calU = \left( U_{\alpha} \right)_{\alpha \in \kappa}$,
    indexed over some set $\kappa$, 
    and we have family of collars
    \begin{align*}
        \calC = \left\{ c_\alpha : \overline{U_\alpha} \times [0,1) \to X \suchthat* \alpha \in \kappa \right\}.    
    \end{align*}
    We write 
    \[
        W_\alpha := c_\alpha( U_\alpha \times [0,1) ),
        \quad 
        \calW := \left\{ W_\alpha \suchthat* \alpha \in \kappa \right\}.
    \]
    Then $\calW$ is a cover of $B$ by open subsets of $X$. 
    First, $B$ is normal because it is paracompact and Hausdorff. 
    We now complete the proof in two steps.
    \begin{enumerate}[1., wide=10pt, itemindent=\parindent, leftmargin=0pt, topsep=0pt, itemsep=0pt]
    \item 
    Since $B$ is strongly paracompact in $X$, 
    the family $\calW$ has a star-finite open refinement 
    $\calO = \left( O_\gamma \right)_{ \gamma \in \mu }$,
    indexed over another set $\mu$. 
    Note that $\left( B \cap O_\gamma \right)_{ \gamma \in \mu }$ is a star-finite open refinement of $\calU$. 
    As $B$ is normal, $B$ has a relatively open cover 
    $\calV = \left( V_\gamma \right)_{ \gamma \in \mu }$ such that 
$\overline{ V_\gamma } \subseteq O_\gamma \cap B$ for each $\gamma \in \mu$.
We understand that $\calV$ is a refinement of $\calU$.
    
    Now, for each $\gamma \in \mu$, 
    we define a local collar $h_\gamma : \overline{ V_\gamma } \times [0,1] \to X$ as follows. 
    First, we fix some $\alpha \in \kappa$ with $V_{\gamma} \subseteq U_{\alpha}$.
    Second, we restrict $c_{\alpha}$ to $\overline{V_{\gamma}} \times [0,1]$.
    Third, we use Corollary~\ref{corollary:restrictedcollarthatavoidsclosedset} on that restriction 
    to ensure that $h_\gamma$ has image within $O_\gamma$.
    It follows that $(h_\gamma)_{\gamma \in \mu}$ is a collar cover of $B$ in $X$
    and the images of those collars form a star-finite collection.
    
    \item 
    So we assume without loss of generality that the images of the collars in $\calC$
    constitute a star-finite collection. 
    In particular, $\calU$ is locally finite then. 
    By the axiom of choice, $\kappa$ can be assumed well-ordered. 
    
    We want to verify the conditions of Theorem~\ref{theorem:collaring:locallyfinitegeneralized}.
    Clearly, $S_{x} := B$ is a relatively open neighborhood of each $x \in B$,
    and since $\calW$ is star-finite, we verify that 
    for each $\alpha \in \kappa$ there are only finitely-many $\beta \in \kappa$ and 
    \begin{align*}
        c_{\alpha}( (S_{x} \cap U_{\alpha}) \times [0,1] )
        \cap 
        W_{\beta}
        = 
        W_{\alpha}
        \cap 
        W_{\beta}
        \neq 
        \emptyset
        .
    \end{align*}
\end{enumerate}
    The conclusion is that $B$ is collared, as had to be shown.
\end{proof}

\begin{remark}
    Theorem~\ref{theorem:collaring:locallyfinitegeneralized} is a generalization of Theorem~\ref{theorem:collaring:locallyfinitecountable}:
    if the collar cover is countable, then we can choose $S_{x} = B$ to satisfy the condition. 
    The utility of this generalization is that it bypasses the machinery for Stone's result (Proposition~\ref{proposition:paracompactimpliessigmadiscrete}) and enables to collar theorem 
    when the closed set is locally collared and strongly paracompact within its Hausdorff ambient space. 
    Nevertheless, we emphasize that with Stone's result available, 
    the as-such weaker result Theorem~\ref{theorem:collaring:locallyfinitecountable} already suffices to prove Theorem~\ref{theorem:collaring:stronglyparacompact}. 
    Therefore, Theorem~\ref{theorem:collaring:locallyfinitegeneralized} is only of expositional interest. 
\end{remark}

\begin{remark} 
    If the collar family $\calC$ in Theorem~\ref{theorem:collaring:locallyfinitegeneralized} is already such that 
    the family of images $\calW$ is star-finite, 
    then the sets $\kappa'(x)$ can be defined alternatively as follows. 
    For every $\alpha \in \kappa$ there exists a finite set $\kappa(\alpha)$
    containing exactly those $\beta \in \kappa$ for which $W_{\alpha} \cup W_\beta \neq \emptyset$.
    We can then set 
$\kappa'(x) = \cup_{ \alpha \in \kappa(x)} \kappa(\alpha)$.
One may choose $S_{x} = B$.
    The remainder of the proof of Theorem~\ref{theorem:collaring:locallyfinitegeneralized} remains unaffected. 
\end{remark}

\section{Bicollar constructions}\label{section:bicollar}

The notion of bicollar is closely related to the notion of collar. Bicollars are also called two-sided collars.
Here, the coordinate does not range over the interval $[0,1]$ but the interval $[-1,1]$. 
Generally speaking, definitions related to collars have analogues for bicollars. 
However, bicollars are most easily constructed by pasting two one-sided collars, 
as this section reviews in more detail. 

A \emph{local bicollar} of $B$ in $X$ is an embedding $c : \overline{U} \times [-1,1] \to X$,
where $U$ is a relatively open subset of $B$,
such that $c(x,0) = x$ for all $x \in \overline{U}$,
such that $c^{-1}(B) = \overline{U} \times \{0\}$,
and such that $c( U \times (-1,1) )$ is open in $X$.
We call $\overline{U}$ the \emph{base} of the local bicollar. 

A local bicollar $c : B \times [-1,1] \to X$ is called a \emph{bicollar} of $B$.
We call $B$ \emph{bicollared} if it has a bicollar.
We say that $B$ is \emph{locally bicollared} in $X$
if every point $x \in B$ has got a relatively open neighborhood $U \subseteq B$
such that there exists a local bicollar $c : \overline{U} \times [-1,1] \to X$ of $B$ in $X$.

A local bicollar is called \emph{strong} if it is a closed embedding, that is, its image is closed in $X$.
Analogously, we define the notions of \emph{strong bicollar} of $B$ in $X$,
as well as the notions of \emph{strongly bicollared} and \emph{locally strongly bicollared} in $X$.

\begin{remark}\label{remark:simplifiedbicollardefinition}
    As with one-sided collars, the definition of bicollar can be simplified.
    Suppose that $B$ is a closed subset of $X$ and that $c : B \times [-1,1] \to X$ is an embedding 
    such that $c(x,0) = x$ for all $x \in B$ and such that $c( B \times (-1,1) )$ is open in $X$.
    Then the condition $c^{-1}(B) = B \times \{0\}$ already holds because $B$ is an embedding,
    and thus $c$ is a bicollar of $B$ in $X$.
\end{remark}

The following concept is fundamental in the discussion of bicollars. 
We call $B$ \emph{two-sided} if $B$ has an open neighborhood $S$ open in $X$
such that for every connected component $O$ of $S$, 
the set $O \setminus B$ has two components, which we call $O^{+}$ and $O^{-}$.
We then say that $B$ is \emph{two-sided in $S$}.
We define 
\begin{align*}
    S^{+} = \bigcup O^{+}, \quad S^{-} = \bigcup O^{-},
\end{align*}
where the unions run over the connected components $O$ of $S$.
This defines a (generally not unique) disjoint partition $S = S^{+} \cup B \cup S^{-}$,
and any such choice of disjoint partition is called a \emph{$B$-separation of $S$}.
We also write 
\[
    S_{0}^{+} = B \cup S^{+}, \quad S_{0}^{-} = B \cup S^{-}.
\]
We remark that the space $X \setminus B$ might still be connected. 

\begin{lemma}\label{lemma:twosidedlocallybicollaredimplieslocallycollared}
    Let $B$ be a closed subset of a topological space $X$ 
    that is two-sided in an open neighborhood $S$. 
    If $B$ is locally (strongly) bicollared, then $B$ is closed and locally (strongly) collared in $B \cup S^{+}$ and $B \cup S^{-}$.
\end{lemma}
\begin{proof} 
    We prove the result in several steps. 
    \begin{enumerate}[1., wide=10pt, itemindent=\parindent, leftmargin=0pt, topsep=0pt, itemsep=0pt]
    \item 
Write $S_{0}^{+} := B \cup S^{+}$ and $S_{0}^{-} := B \cup S^{-}$.
    We observe that $S_{0}^{+}$ and $S_{0}^{-}$ are closed in $S$. 
    Consider now a collection $\calC$ of (strong) local bicollars whose bases cover $B$:
    \[
        \calC = \left\{ c_{\alpha} : \overline{U_\alpha} \times [-1,1] \to X \suchthat* \alpha \in \kappa \right\}
        .
    \]
    We define $c^{+} := c_{|\overline{U_\alpha} \times [0,1]}$ and $c^{-} := c_{|\overline{U_\alpha} \times [-1,0]}$
    for each $c \in \calC$. 
    
    \item 
    Let $c_{\alpha} \in \calC$. 
    Then $\overline{U_{\alpha}}$ is (strongly) collared in the images of $c^{+}_{\alpha}$ and $c^{-}_{\alpha}$. 
    Using Corollary~\ref{corollary:restrictedcollarthatavoidsclosedset},
    we obtain (strong) collars $c^{+}_{\alpha,\#}$ and $c^{-}_{\alpha,\#}$
    whose images lie within the images of $c^{+}_{\alpha}$ and $c^{-}_{\alpha}$, respectively,
    but are disjoint from the closed set $X \setminus S$
    and which look like $c^{+}_{\alpha}$ and $c^{-}_{\alpha}$, respectively,
    over a sufficiently small neighborhood of any $x \in \overline{U_\alpha}$. 
    We join them to a (strong) bicollar 
    \begin{align*}
        c_{\alpha,\#} : \overline{U_{\alpha}} \times [-1,1] \to X, \quad (x,t) = 
        \begin{cases} 
            c^{+}_{\alpha,\#}(x,t)  & \text{ if } t \geq 0,
            \\
            c^{-}_{\alpha,\#}(x,t) & \text{ if } t < 0,
        \end{cases}
    \end{align*}
    whose image lies in the image of $c_{\alpha}$ and is disjoint from $X \setminus S$. 
    We may thence assume without loss of generality that the members of $\calC$
    map into $S$.
    
    \item 
    Consider any $x \in \overline{U_{\alpha}}$. 
    The image $c_{\alpha}\left( \{x\} \times [0,1] \right)$ is connected,
    and it is either a subset of $S_{0}^{+}$ or of $S_{0}^{-}$. 
    We introduce the disjoint relatively open subsets of $\overline{U_{\alpha}} \times [0,1]$:
    \begin{align*}
        O_{\alpha}^{+} := (\overline{U_{\alpha}} \times  (0,1]) \cap c_{\alpha}^{-1}(S^{+}),
        \quad 
        O_{\alpha}^{-} := (\overline{U_{\alpha}} \times  (0,1]) \cap c_{\alpha}^{-1}(S^{-}).
    \end{align*}
    For any $t \in (0,1]$, 
    if $(x,t) \in O_{\alpha}^{+}$, then $\{x\} \times (0,1] \subseteq O_{\alpha}^{+}$,
    and 
    if $(x,t) \in O_{\alpha}^{-}$, then $\{x\} \times (0,1] \subseteq O_{\alpha}^{-}$.
    Consequently, we have a partition of $U_{\alpha}$ into disjoint sets 
    \begin{align*}
      E^{+} := \left\{ x \in \overline{U_{\alpha}} \suchthat* \{x\} \times (0,1] \subseteq O_{\alpha}^{+} \right\},
      \\
      E^{-} := \left\{ x \in \overline{U_{\alpha}} \suchthat* \{x\} \times (0,1] \subseteq O_{\alpha}^{-} \right\}.
    \end{align*}
    Notice that 
    \begin{align*}
        O_{\alpha}^{+} = E^{+} \times (0,1],
        \quad 
        O_{\alpha}^{-} = E^{-} \times (0,1].
    \end{align*}
    It is now easily seen that $E^{+}$ and $E^{-}$ are disjoint and open. 
    We define 
    \begin{align*}
        c_{\alpha,\Uparrow} : \overline{U_{\alpha}} \times [-1,1] \to X, \quad (x,t) = 
        \begin{cases} 
            c_{\alpha}(x, t) & \text{ if } x \in E^{+},
            \\
            c_{\alpha}(x,-t) & \text{ if } x \in E^{-}.
        \end{cases}
    \end{align*}
    This is a (strong) local bicollar of $B$ in $X$ again.
    Consequently, we may assume without loss of generality that each $c_{\alpha} \in \calC$ satisfies
    \begin{align*}
        c_{\alpha}( \overline{U_{\alpha}} \times [ 0,1] ) \subseteq S^{+}_{0},
        \quad 
        c_{\alpha}( \overline{U_{\alpha}} \times [-1,0] ) \subseteq S^{-}_{0}.
    \end{align*}
    Clearly, 
    $c_{\alpha}^{+}$ and $c_{\alpha}^{-}$ are local homeomorphisms,
    they are the identity over $\overline{U_{\alpha}}$,
and their images intersect with $B$ only along $\overline{U_{\alpha}}$.
    Moreover, their images are open in $S_{0}^{+}$ and in $S_{0}^{-}$, respectively. 
    \end{enumerate}
    In summary, $B$ is closed and locally (strongly) collared both in $S_{0}^{+}$ and in $S_{0}^{-}$.
\end{proof}

\begin{lemma}
    Let $B$ be a closed subset of a topological space $X$ 
    that is two-sided in an open neighborhood $S$. 
    If $B$ is (strongly) collared in $B \cup S^{+}$ and $B \cup S^{-}$,
    then $B$ is (strongly) bicollared in $X$.
\end{lemma}
\begin{proof}    
    Consider the two sets $S_{0}^{+} = B \cup S^{+}$ and $S_{0}^{-} = B \cup S^{-}$,
    which are both closed in $S$.   
    By assumption, we have collars 
    \begin{gather*}
        h^{+} : B \times [0,1] \rightarrow S_{0}^{+}, 
        \quad 
        h^{-} : B \times [0,1] \rightarrow S_{0}^{-}.
    \end{gather*}
    In particular, $h^{+}(x,0) = h^{-}(x,0) = x$ for any $x \in B$. 
    We define 
    \begin{align*}
        h : B \times [-1,1] \to X, \quad (x,t) = 
        \begin{cases} 
            h^{+}(x,t)  & \text{ if } t \geq 0,
            \\
            h^{-}(x,-t) & \text{ if } t < 0.
        \end{cases}
    \end{align*}
    Let $W \subseteq S$, $W^{+} \subseteq S_{0}^{+}$ and $W^{-} \subseteq S_{0}^{-}$ denote the images of the embeddings $h$, $h^{+}$ and $h^{-}$, respectively.
    Then $B = W^{+} \cap W^{-}$ and $W = W^{+} \cup W^{-}$.
    If $h^{+}$ and $h^{-}$ have closed images, then so does $h$. 
    The sets $h( B \times (0,1) )$ and $h( B \times (-1,0) )$ are open in $S^{+}$ and $S^{-}$, respectively,
    and thus they are open in $W$, which means that $W^{+}$ and $W^{-}$ are both closed in $W$.
    
    Obviously, $h$ is injective. 
    The sets $B \times [0,1]$ and $B \times [-1,0]$ are closed in $B \times [-1,1]$,
    and so the pasting lemma implies that $h$ is continuous. 
    
    If $O \subseteq B \times [-1,1]$ is open, 
    then $h( O \cap (B \times [0,1]) )$ and $h( O \cap (B \times [-1,0]) )$ are open in $W^{+}$ and $W^{-}$, respectively.
    Since the latter two sets are closed in $W$, and since $B = W^{+} \cap W^{-}$, 
    Lemma~\ref{lemma:joiningrelativeopensets} now shows that $h$ maps open subsets of $B \times [-1,1]$ onto sets relatively open within its image.
    
    Finally, $h( B \times [0,1) )$ and $h( B \times (-1,0] )$ are open in $S_{0}^{+}$ and $S_{0}^{-}$, respectively. 
    The latter two sets are closed in $S$, and we have $B = S_{0}^{+} \cap S_{0}^{-}$.
    Another application of Lemma~\ref{lemma:joiningrelativeopensets} shows 
    that $h( B \times (-1,1) )$ is open in $S$. 
    Since $S$ is open in $X$, 
    it follows that $h( B \times (-1,1) )$ is open in $X$ as well. 
\end{proof}

\begin{theorem}
    Let $B$ be a locally bicollared two-sided closed subset of a Hausdorff space $X$.
    If $B$ is paracompact in $X$, then $B$ is bicollared.
\end{theorem}
\begin{proof}
    We combine the preceding two lemmas and Theorem~\ref{theorem:collaring:paracompactinambient}.
\end{proof}

\section{Metric spaces and Lipschitz collars}\label{section:lipschitz}

Up to now, we have worked within the category of topological spaces. 
The second big topic of this manuscript is the collar theorem within the Lipschitz category. 
Here, the morphisms between metric spaces are locally Lipschitz mappings.
We prove the Lipschitz collar theorem when the base set is closed, as already proven by Luukkainen and V\"ais\"al\"a. 
In addition, our constructive approach enables us to prove that the global collar is bi-Lipschitz 
under additional assumptions on the initial collection of local collars,
and we provide explicit estimates. 
\\

Let $X$ be a metric space.
In what follows, all metrics are denoted by $\dist$. 
We also write $\Ball(z,r)$ for the open ball of radius $r \in (0,\infty]$ centered at $z \in X$,
and we write $\Ball(S,r) = \cup_{z \in S}\Ball(z,r)$ for any set $S \subseteq X$. 
Whenever $X$ and $Y$ are metric spaces, $X \times Y$ carries the product metric\footnote{Of course, other choices of product metrics are possible as well,
but this is the most convenient one for our purposes.}
\[
    \dist\left( (x,y), (x',y') \right) := \dist(x,x') + \dist(y,y').
\]
All metric spaces are Hausdorff, regular, normal, collectionwise normal, and paracompact.
If $\calU$ is an open cover of $X$, then $\delta > 0$ is called a \emph{Lebesgue number} of the cover 
if for all $x \in X$, the ball $\Ball(x,\delta)$ is contained in some member of $\calU$. 

Suppose that $f \colon U \to Y$ is a continuous mapping from $X$ into another metric space $Y$.
Given any subset $U \subseteq X$,
we call $f$ \emph{Lipschitz over $U$} with Lipschitz constant $\Lip(f,U) > 0$ if 
\begin{align*}
    \dist\left( f(x), f(x') \right) \leq \Lip(f,U) \cdot \dist\left(x,x'\right), \quad x, x' \in U.
\end{align*}
We call $f : X \to Y$ \emph{locally Lipschitz} if every point in $X$ has an open neighborhood over which $f$ is Lipschitz. 
If $f : X \to Y$ is a homeomorphism, 
then we call $f$ a \emph{lipeomorphism} if $f$ and its inverse are locally Lipschitz,
and we call $f$ \emph{bi-Lipschitz} if $f$ and its inverse are Lipschitz.
Analogously, if $f : X \rightarrow Y$ is an embedding, 
then $f$ is called \emph{LIP} if $f : X \rightarrow f(X)$ is a lipeomorphism,
and we call it (with some abuse of terminology) \emph{bi-Lipschitz} if $f : X \rightarrow f(X)$ is bi-Lipschitz. 
Whenever $f : X \to f(X)$ is bi-Lipschitz, we abbreviate 
\begin{align*}
    \LipInv(f,X) := \Lip(f^{-1},f(X)),
    \quad
\biLip(f,X) := \Lip(f,X) \cdot \LipInv(f,X)
    .
\end{align*}
We may further abbreviate those constants by $\Lip(f)$, $\LipInv(f)$, and $\biLip(f)$ whenever there is no danger of ambiguity. 
\\

We begin with the discussion of the collar theorem. 
First, we establish an important fundamental result, namely the existence of locally Lipschitz partitions of unity,
which is of interest on its own. 
Subsequently, we extend a few auxiliary results from the topological case to the Lipschitz case. 

\begin{proposition}\label{proposition:lipschitzpartitionofunity}
    Let $\calU = ( U_\alpha )_{\alpha \in \kappa}$ be a locally finite open cover of $X$.
    Then there exists a locally Lipschitz partition of unity $( \lambda_\alpha )_{\alpha \in \kappa}$ subordinate to $\calU$.
    
    In addition, suppose that $\calU$ has Lebesgue number $\delta > 0$
    and that each point of $X$ intersects at most $N \in \bbN$ members of $\calU$.
    Given any fixed $\delta_0 \in (0,\delta)$, 
    we can assume for all $\alpha \in \kappa$ that $\Ball(z,\delta-\delta_0) \subseteq U_{\alpha}$ for all $z \in \supp(\lambda_\alpha)$
    and that 
    \begin{gather*}
        \Lip\left( \lambda_\alpha \right) \leq N \delta_{0}^{-1}
        ,
        \quad 
        \Lip\left( \lambda_\alpha, U_{\alpha} \right) \leq (N-1) \delta_{0}^{-1}
        .
    \end{gather*}
    Moreover, for any subset $\mu \subseteq \kappa$ we then have 
    \begin{gather*}
        \Lip\left(
            \textstyle\sum\nolimits_{\beta\in\mu}\lambda_\beta
        \right)
        \leq 
        N \delta_{0}^{-1}
        .
    \end{gather*}
\end{proposition}
\begin{proof} 
    We let $\calV = \left( V_{\alpha} \right)_{\alpha \in \kappa}$ be an open cover of $X$
    satisfying $\overline{V_\alpha} \subseteq U_{\alpha}$ for each $\alpha \in \kappa$. 
    We define\footnote{If $V_{\alpha} \neq X$, we can simply define $f_\alpha(x) := \dist( x, X \setminus V_\alpha )$.} $f_\alpha(x) := \min( 1, \dist( x, X \setminus V_\alpha ) )$ and 
    $\lambda_\alpha(x) := f_\alpha(x) / \textstyle\sum\nolimits_{ \beta \in \kappa } f_\beta(x)$.
Each $f_\alpha$ is a non-negative function with support equal to $\overline V_\alpha$, positive over $V_\alpha$,
    and having Lipschitz constant $1$.
    The function $\lambda_{\alpha}$ is well-defined because $\calV$ is a locally finite cover of $X$.
    Any $x \in X$ has a neighborhood $U$ such that there exists $\alpha \in \kappa$ satisfying $f_{\alpha|U} \geq \epsilon$
    for some $\epsilon > 0$.
    Consequently, each $\lambda_{\alpha}$ is locally Lipschitz~\cite[Section~2, see also Theorem~5.3]{luukkainen1977elements}.
    
    Now suppose that $\delta$ and $N$ are as in the second part of the proposition, and that $\delta_0 \in (0,\delta)$.
    We then define $V_{\alpha}$ as the set of those points in $U_{\alpha}$
    whose distance to $X \setminus U_{\alpha}$ is strictly more than $\delta - \delta_0 > 0$.
    Clearly, $\calV$ is a locally finite open cover with Lebesgue number $\delta_{0}$.
    We have $\Ball(z,\delta - \delta_0) \subseteq U_{\alpha}$ for all $z \in \supp(\lambda_\alpha)$, $\alpha \in \kappa$.
    
    Given any subset $\mu \subseteq \kappa$, 
    we henceforth abbreviate
    \begin{align*}
        f_{\mu} := \textstyle\sum\nolimits_{ \beta \in \mu } f_\beta,
        \quad 
        \lambda_{\mu} := \textstyle\sum\nolimits_{ \beta \in \mu } \lambda_\beta.
    \end{align*}
    These are well-defined functions. 
    When $x,y \in X$ and $\alpha \in \kappa$, then $f_{\alpha}(x) = 0 \leq f_{\alpha}(y)$ or $f_{\alpha}(x) \leq f_{\alpha}(y) + \dist(x,y)$. 
    Therefore, $f_{\mu}(x) \leq f_{\mu}(y) + N \dist(x,y)$ holds for any $\mu \subseteq \kappa$.
    In particular, $\Lip(f_{\mu}) \leq N$.
    
    Given $x, y \in X$ and $\mu \subseteq \kappa$, we observe  
    \begin{align*}
        \lambda_\mu(x) - \lambda_\mu(y)
        &
        =
\frac{ \left( f_\mu(x) - f_\mu(y) \right) f_\kappa(y) - f_\mu(y) \left( f_\kappa(x) - f_\kappa(y) \right) }
        { f_\kappa(x) f_\kappa(y) }
        \\&\;
        =
        \frac{ 
            \left( f_\mu(x) - f_\mu(y) \right) f_{\kappa\setminus\mu}(y)
            -
            f_\mu(y) \left( f_{\kappa\setminus\mu}(x) - f_{\kappa\setminus\mu}(y) \right) 
        }{ f_\kappa(x) f_\kappa(y) }
        .
    \end{align*}
    We obtain the upper bound 
    \begin{align*}
        \left| \lambda_\mu(x) - \lambda_\mu(y) \right| 
        &
        \leq
        \frac{ 
            \Lip( f_{\mu} ) f_{\kappa\setminus\mu}(y) 
            +
            \Lip( f_{\kappa\setminus\mu} ) f_{\mu}(y) 
        }{
            f_\kappa(x) f_\kappa(y) 
        }
        \dist(x,y) 
        \\&
        \leq
        \max\left( 
            \Lip( f_{\mu} )
            ,
            \Lip( f_{\kappa\setminus\mu} )
        \right)
        \delta_{0}^{-1}
        \dist(x,y) 
        .
    \end{align*}
    Consider now the special case where $\mu = \{\alpha\}$ is a singleton.
In the special case $x, y \in V_{\alpha}$, one easily estimates 
    \begin{align*}
        \left| \lambda_\alpha(x) - \lambda_\alpha(y) \right|
        &
        \leq
        \left( 
            N-1
        \right)
        \delta_{0}^{-1}
        \dist(x,y) 
        .
    \end{align*}
    The desired inequalities are now evident.  
\end{proof}

\begin{lemma}\label{lemma:lipschitz:cutfunctionthatavoidsclosedset}
    Let $X$ be a metric space, let $B \subseteq X$ be a closed subset,
    and let $A \subseteq X$ be a closed set with $A \cap B = \emptyset$. 
    If $c : B \times [0,1] \to X$ be a LIP collar of $B$ in $X$,
    then there exists a locally Lipschitz function $d : B \to (0,1]$ such that
    for all $x \in B$ and $t \in [0,d(x)]$ we have $c(x,t) \notin A$.
In particular, 
    \begin{align*}
        h : \overline{U} \times [0,1] \to X, \quad (x,t) \mapsto c(x,t d(x))
    \end{align*}
    is a LIP collar. 
\end{lemma}
\begin{proof}
    We follow the proof Lemma~\ref{lemma:cutfunctionthatavoidsclosedset};
    and then use Proposition~\ref{proposition:lipschitzpartitionofunity}
    to ensure that the partition of unity in that proof is locally Lipschitz.
    Then $d$ is a locally Lipschitz non-negative function. 
    Clearly, $h$ as above is a LIP collar. 
\end{proof}

\begin{lemma}\label{lemma:lipschitz:collaringdiscrete}
    Let $B$ be a closed subset of the metric space $X$.
    If
    \begin{align*}
        \calC = \left( c_{\alpha} : \overline{U_{\alpha}} \times [0,1] \to X \right)_{\alpha \in \kappa} 
    \end{align*}
    is a family of LIP local collars of $B$ in $X$ such that the family $(U_{\alpha})_{\alpha \in \kappa}$ is discrete, 
    then $B$ has a LIP local collar
    \begin{align*}
        c : \bigcup_{\alpha \in \kappa} \overline{U_\alpha} \times [0,1] \to X.
    \end{align*}
\end{lemma}
\begin{proof} 
    This is the same proof as for Lemma~\ref{lemma:collaringdiscrete},
    except that we invoke Lemma~\ref{lemma:lipschitz:cutfunctionthatavoidsclosedset}
    instead of Corollary~\ref{corollary:restrictedcollarthatavoidsclosedset}.
    By construction, the union $c$ is a LIP embedding. 
\end{proof}

The following technical observation will be used repeatedly in what follows.
\begin{lemma}
    Let $a_{1},a_{2},c_{1},c_{2} \in \bbR$ with $a_{1} \leq c_{1}$ and $c_{2} \leq a_{2}$. Then 
    \[ |a_{1}-c_{1}| + |c_{2}-a_{2}| \leq |a_{1}-a_{2}| + |c_{1}-c_{2}|. \]
\end{lemma}
\begin{proof}
    An easy case distinction.
\end{proof}

This section's first main result is an extension of Theorem~\ref{theorem:collaring:locallyfinitecountable}. 
On the one hand, we show that a locally LIP collared closed set is globally LIP collared. 
On the other hand, we prove bi-Lipschitz estimates; 
however, the assumptions needed for that are considerably stronger and are commented on below.

\begin{theorem}\label{theorem:lipschitz:locallyfinitecountable}
    Let $B$ be a closed subset of a metric space $X$.
    Let 
    \begin{align*}
        \calC = \left( c_{i} : \overline{U_{i}} \times [0,1] \to X \right)_{i \in \bbN} 
    \end{align*}
    be a collection of local LIP collars of $B$ in $X$
    such that the family $\calU = (U_{i})_{i \in \bbN}$ is locally finite. 
    Let $\left( \lambda_{i} : B \to [0,1] \right)_{i \in \bbN}$ be a partition of unity subordinate to $\calU$. 
    
    \begin{enumerate}[1., wide=10pt, itemindent=\parindent, leftmargin=0pt, topsep=0pt, itemsep=0pt]
    \item 
    If each $c_{i}$ is a LIP embedding and each $\lambda_{i}$ is locally Lipschitz, 
    then there exists a collar $c : B \times [0,1] \to X$ of $B$ in $X$ 
    that is a LIP embedding. 
    \item
    We have the estimate 
    \begin{align*}
        \Lip\left( h, B \times [0,1] \right) 
        &
        \leq 
C
        ( 1 +   L_{\Sigma} / 2 )
        ( 1 + 2 L_{\Sigma}     )
        \max\left( 1 + 2 L , 1 + \zeta^{-1} \right)^{N}
        C_{b}^{N}
        ,
        \\ 
\LipInv\left( h, B \times [0,1] \right) 
        &\leq
C
        ( 1 + L_{\Sigma} / 2 )
        C_{b}^{2N} 
        \max\left( 
            1 + \frac C \zeta + \frac L 2
        \right)^{N}
        \max\left( 3 + \frac  1 \zeta + 3 L \right)^{N} ,
    \end{align*}
    where 
    \begin{gather*}
        L := \max_{i \in \bbN}\Lip(\lambda_i,U_{i}), 
        \quad 
        L_{\Sigma} := \max_{i \in \bbN}\Lip\left( \lambda_{1}+\lambda_{2}+\cdots+\lambda_{i} \right), 
        \\ 
        C := \max_{ i \in \bbN } \Lip( c_i ), 
        \quad  
        C_b := \max_{ i \in \bbN } \biLip( c_i ),
    \end{gather*}
    where $\zeta > 0$ is such that $\Ball(z,\zeta) \subseteq c_{i}\left( \overline{U_i} \times [0,1] \right)$ 
    for all $z \in c_{i}\left( \supp(\lambda_i) \times [0,0.75] \right)$ and $i \in \bbN$,
    and where $N \in \bbN$ satisfies the following property:
    every $x \in B$ has a relatively open neighborhood ${S}_x \subseteq B$ 
    such that for each $i \in \bbN$ there are at most $N$ different $j \in \bbN$ with $j \leq i$ and 
    \[
        c_{i}( ({S}_x \cap \overline{U_i}) \times [0,1] ) \cap c_{j}( \overline{U_j} \times [0,1] ) \neq \emptyset.
    \]
    \end{enumerate}
\end{theorem}
\begin{proof}
    This proof is to be read as a continuation of the proof of Theorem~\ref{theorem:collaring:locallyfinitecountable}.
    \begin{enumerate}[1., wide=10pt, itemindent=\parindent, leftmargin=0pt, topsep=0pt, itemsep=0pt]
        \item The partition of unity is locally Lipschitz by assumption. 
        The mappings 
        \begin{align*}
            q_i : Q_{\Sigma,i} \setminus O_{\Sigma,i-1} \to Q_{i,\mytent}, 
            \quad 
            (x,t) \mapsto \left( x, ( t - \lambda_{\Sigma,i-1}(x) ) / 2 \right)
        \end{align*}
        are lipeomorphisms. 
        Elementary calculations show that $\Xi_i$ is lipeomorphism.
        Clearly, we have $\dist( (x,t), \Xi_i(x,t) ) \leq 1$ 
        for all $(x,t) \in \overline{U_i} \times [0,1]$. 
        
        If, additionally, the partition of unity functions are Lipschitz, 
        then $q_{i}$ and $\Xi_{i}$ are bi-Lipschitz, and we have 
\begin{gather*}
            \max\left(
                \LipInv( q_i )
                ,
                \Lip   ( q_i )
            \right)
            \leq
            1 + \frac 1 2 \sum_{j=1}^{i-1} \Lip( \lambda_j )
            ,
            \\
\Lip( \Xi_i )
            \leq
            1 + 2 \Lip(\lambda_i,U_{i})
            ,
            \quad 
            \LipInv( \Xi_i )
            \leq
            3 + 3 \Lip(\lambda_i,U_{i})
            .
        \end{gather*}

        \item 
        By assumption, $c_i$ is a LIP embedding, and so is $g_i$. 
        
        If, additionally, the partition of unity is Lipschitz and $c_i$ is bi-Lipschitz, 
        then $g_i$ is bi-Lipschitz, as we now show.
        Write $W_i := c_i\left( \overline{U_i} \times [0,1] \right)$.
        We know that $g_i$ is the identity outside of $W_i$,
        and that 
        \begin{align*}
            \Lip(g_i,W_i) 
            \leq 
            \Lip(\Xi_i) 
            \biLip \left( c_i \right), 
            \quad 
            \LipInv(g_i,W_i) 
            \leq 
            \LipInv(\Xi_i) 
            \biLip \left( c_i \right) 
            .
        \end{align*}
        Now suppose that $z,w \in X$ such that $z \notin W_i$ and $g(w) \neq w$. 
        We make the additional assumption that there exists $\zeta_{i} > 0$ such that 
        \begin{align*}
            \Ball\left( c_{i}\left( \supp(\lambda_{i}) \times [0,0.75] \right), \zeta_{i} \right) \subseteq W_{i}.
        \end{align*}
        Thus, knowing $g_{i}(z) = z$ and $w, g_{i}(w) \in c_{i}\left( \supp(\lambda_{i}) \times [0,0.75] \right)$, we find 
        \begin{align*}
            \dist(z,w) \geq \zeta_i
            ,
            \quad 
            \dist(g_{i}(z),g_{i}(w)) \geq \zeta_i
            .
        \end{align*}
        Clearly, $\dist( w, g_i(w) ) \leq \biLip(c_{i})$.
        Consequently, 
        \begin{align*}
            \dist( g_i(z), g_i(w) )
            &\leq 
            \dist( z, w ) + \dist( w, g_i(w) )
            \\&\leq 
            \dist( z, w ) + \zeta_i^{-1} \biLip(c_i) \dist( z, w )
            ,
            \\
            \dist( z, w )
            &\leq 
            \dist( g_i(z), g_i(w) ) + \dist( w, g_i(w) )
            \\&\leq 
            \dist( g_i(z), g_i(w) ) + \zeta_i^{-1} \biLip(c_i) \dist( g_i(z), g_i(w) )
            .
        \end{align*}
        This establishes the bi-Lipschitz estimates for $g_i$:
        \begin{align*}
            \Lip(g_i) 
            \leq 
            \max\left( 1 + \zeta^{-1} \biLip\left( c_i \right), \Lip(\Xi_i) \biLip\left( c_i \right) \right)
            \leq 
            \left(1+\zeta^{-1},\Lip(\Xi_i)\right)
            \biLip \left( c_i \right)
            , 
            \\ 
            \LipInv(g_i) 
            \leq 
            \max\left( 1 + \zeta^{-1} \biLip\left( c_i \right), \LipInv(\Xi_i) \biLip\left( c_i \right) \right)
            \leq 
            \left(1+\zeta^{-1},\LipInv(\Xi_i)\right)
            \biLip\left( c_i \right) 
            .
        \end{align*}

        \item 
        Before we begin with the main part of the proof, 
        we establish another series of auxiliary inequalities. 
        If ${z} \in \overline{U_{i}}$ and $\rho > 0$ are such that $\lambda_{i}$ is Lipschitz over $\Ball({z},\rho) \cap \overline{U_{i}}$, 
        consider any $x,y \in \overline{U_{i}}$ and $s, t \in [0,1]$ with $s \leq \lambda_{i}(x)/2$ and $\lambda_{i}(y)/2 \leq t$.
        One easily checks that 
        \begin{align*}
            &
            \left| s - \lambda_{i}(x)/2 \right| + \left| \lambda_{i}(y)/2 - t \right|
            \leq 
            \left| s - t \right| + \frac 1 2 \left| \lambda_{i}(x) - \lambda_{i}(y) \right|
            .
        \end{align*}
        With that, we estimate 
        \begin{align*}
                \dist( (x,s), (y,t) )
                &
                \leq
                \dist\left( (x,s), (x,\lambda_{i}(x)/2) \right)
                +
                \dist\left( (x,\lambda_{i}(x)/2), (y,t) \right)
                \\& 
                =
                \dist(x,y)
                +
                | t - \lambda_{i}(x)/2 |
                +
                | \lambda_{i}(x)/2 - s |
                \\& 
                \leq 
                \dist(x,y)
                +
                | t - s |
                +
                \frac 1 2 
                | \lambda_{i}(y) - \lambda_{i}(x) |
                \\& 
                \leq 
                \left( 1 + \frac 1 2 \Lip\left( \lambda_{i},\Ball(z,\rho) \cap \overline{U_{i}} \right) \right) \dist( (x,s), (y,t) )
                .
        \end{align*}
If, additionally, $\lambda_{i}$ is Lipschitz, 
        then we can choose $\rho = \infty$,
        and then that inequality applies for all $x,y \in \overline{U_{i}}$.
        
        \item 
        This translates to an inequality within the image of the collar. 
        Since $g_{i}$ is a LIP embedding,
        for every $z \in \overline U_{i}$ there exists $\rho > 0$ and $\epsilon > 0$
        such that for all $x,y \in \Ball(z,\rho) \cap \overline U_{i}$ and $s,t \in [0,1]$ with 
        \begin{align*}
            \lambda_{i}(z)/2 - \epsilon 
            \leq
            s
            \leq
            \lambda_{i}(z)/2 
            \leq
            t
            \leq
            \lambda_{i}(z)/2 + \epsilon 
        \end{align*}
        we have, writing $V := (\Ball(z,\rho) \cap \overline U_{i}) \times [\lambda_{i}(z)/2 - \epsilon,\lambda_{i}(z)/2 + \epsilon ]$, the inequality
        \begin{align*}
                &
                \dist\left( c_{i}(x,s), c_{i}(x,\lambda_{i}(x)/2) \right)
                +
                \dist\left( c_{i}(x,\lambda_{i}(x)/2), g_{i}(c_{i}(y,t)) \right)
                \\&\; 
                \leq 
                \biLip\left(c_{i}, V \right)
                \left( 1 + \frac 1 2 \Lip\left( \lambda_{i},\Ball(z,\rho) \cap \overline{U_{i}} \right) \right)
                \dist\left( c_{i}(x,s), g_{i}(c_{i}(y,t)) \right)
                .
        \end{align*}
        If, additionally, $\lambda_{i}$ is Lipschitz and $c_{i}$ is bi-Lipschitz, 
        so that $g_{i}$ is bi-Lipschitz as well, 
        then the previous inequality holds for any $x,y \in \overline U_{i}$ and $s,t \in [0,1]$.
        Now consider any $z \in X$ and $(x,s) \in \overline U_{i} \times [0,1]$ with $0 \leq s \leq \lambda_{i}(x)$.
        If $z$ is in the image of $c_{i}$, then so is $g_{i}(z)$, and 
        \begin{align*}
            &
            \dist\left( c_{i}(x,s), c_{i}(x,\lambda_{i}(x)/2) \right)
            +
            \dist\left( c_{i}(x,\lambda_{i}(x)/2), g_{i}(z) \right)
            \\&\; 
            \leq 
            \biLip\left(c_{i}\right) \left( 1 + \frac 1 2 \Lip(\lambda_{i},\overline U_{i}) \right)
            \dist\left( c_{i}(x,s), g_{i}(z) \right)
            .
        \end{align*}
        Otherwise, suppose that $\zeta_{i} > 0$ is such that $\Ball( c_{i}(x,s), \zeta_{i} )$ is contained in the image of $c_{i}$.
        We easily estimate 
        \begin{align*}
            \dist\left( c_{i}(x,s), c_{i}(x,\lambda_{i}(x)/2) \right)
            &
            \leq 
            \Lip(c_{i}) \left| \frac{ \lambda_{i}(x) }{2} - s \right|
\leq 
            \frac{ \Lip(c_{i}) }{2}
            \zeta_{i}^{-1}
            \dist\left( c_{i}(x,s), g_{i}(z) \right)
            ,
        \end{align*}
        which leads to 
        \begin{align*}
            &
            \dist\left( c_{i}(x,s), c_{i}(x,\lambda_{i}(x)/2) \right)
            +
            \dist\left( c_{i}(x,\lambda_{i}(x)/2), g_{i}(z) \right)
            \\&\; 
            \leq 
            2
            \dist\left( c_{i}(x,s), c_{i}(x,\lambda_{i}(x)/2) \right)
            +
            \dist\left( c_{i}(x,s), g_{i}(z) \right)
            \\&\; 
            \leq 
\left( 1 + \zeta_{i}^{-1} \Lip(c_{i}) \right)
            \dist\left( c_{i}(x,s), g_{i}(z) \right)
            .
        \end{align*}
        To summarize, writing 
        \[
            \vartheta_{i} :=             \max\left(
                1 + \zeta_{i}^{-1} \Lip(c_{i})
                ,
                \biLip\left(c_{i}\right) \left( 1 + \niceonehalf \Lip(\lambda_{i}) \right)
            \right),
        \]
        we have derived the auxiliary inequality  
        \begin{align*}
            &
            \dist\left( c_{i}(x,s), c_{i}(x,\lambda_{i}(x)/2) \right)
            +
            \dist\left( c_{i}(x,\lambda_{i}(x)/2), g_{i}(z) \right)
\leq 
            \vartheta_{i}
            \dist\left( c_{i}(x,s), g_{i}(z) \right)
            .
        \end{align*}

        \item 
        We begin with the main part of the proof. 
        We recall that $h$ equals $g_{1} \circ \cdots \circ g_{i-1} \circ c_{i} \circ q_{i}$ 
        over the subspace $Q_{\Sigma,i} \setminus O_{\Sigma,i-1}$.
        In particular, whenever $(z,r) \in O_{\Sigma,i} \setminus Q_{\Sigma,i-1}$,
        then $h$ is bi-Lipschitz over an open neighborhood of $(z,r)$. 
        
        If, additionally, the collars are uniformly bi-Lipschitz and the partition of unity functions are uniformly Lipschitz, then $h$ is bi-Lipschitz over $Q_{\Sigma,i} \setminus O_{\Sigma,i-1}$.

        \item 
        We now see that it remains to consider the case where $(z,r) \in B \times [0,1]$ lies on the graph of one of the functions $\lambda_{\Sigma,i}$. 
        We choose $\beta \in (0,\infty]$ and a relatively open interval $I \subseteq [0,1]$ 
        that contains $r$ such that the following conditions hold:
        \begin{itemize}
            \item
            All the partition of unity functions are Lipschitz over $\Ball(z,\beta)$. 
            This is possible because the partition of unity functions are locally Lipschitz and their supports are a locally finite cover of $B$. 
            \item
            The graph of each $\lambda_{\Sigma,j}$, $j \in \bbN$, 
            over ${S} \cap \Ball(z,\beta)$ is either disjoint from $\Ball(z,\beta) \times I$
            or a subset of that set. 
            This is possible because the partition of unity is locally finite and locally Lipschitz.
            \item
            $h$ is bi-Lipschitz over $V \cap Q_{\Sigma,j} \setminus O_{\Sigma,j-1}$, $j \in \bbN$.
            This is possible because $h$ is a LIP embedding over $Q_{\Sigma,j} \setminus O_{\Sigma,j-1}$, $j \in \bbN$, and the partition of unity is locally finite.
        \end{itemize}
        If, additionally,
        the collars are bi-Lipschitz and the partition of unity functions are Lipschitz,
        then we can choose $\beta = 0$ and $I = [0,1]$.
        
        \item Let ${S} \subseteq B$ be a relatively open neighborhood of $z$ and let $N_{S} \in \bbN$ satisfy the following condition:
        ${S}$ intersects only finitely many members of $\calU$
        and 
        for every $i \in \bbN$ and $x \in S \cap U_{i}$, there are at most $N_{S}$ different $j \in \bbN$ with $j \leq i$ for which 
        \[
            c_{i}( \{x\} \times [0,1] ) \cap c_{j}( U_j \times [0,1] ) \neq \emptyset.
        \]
        Note that every $z \in B$ has a neighborhood ${S}$ such the condition is satisfied for some sufficiently large choice of $N_{S} \in \bbN$. 
        Indeed, since $\calU$ is locally finite and countable, we can choose ${S} = U_{i}$ with $i$ maximal; in that case, $N_{S} = i$.

        \item 
        With $\rho > 0$ and $I \subseteq [0,1]$ as well as $S$ as above, 
        we let $V$ be the following open neighborhood of $(z,r)$:
        \[
            V = ( {S} \cap \Ball(z,\beta) ) \times I.
        \]
        Since $S$ intersects only finitely many members of $\calU$, 
        there exists a maximal $m \in \bbN_0$ such that $\Ball(z,\beta) \times I$ 
        lies in the hypergraph of $\lambda_{\Sigma,m}$.
        Analogously, 
        there exists a minimal $n \in \bbN$ such that $\Ball(z,\beta) \times I$ 
        lies in the hypograph of $\lambda_{\Sigma,n}$.
        
        \item 
        We now address the (local) Lipschitz estimates of $h$.
        Let $(x,s), (y,t) \in V$. 
        If there exists $j \in \bbN_{0}$ such that $(x,s), (y,t) \in V \cap Q_{\Sigma,j} \setminus O_{\Sigma,j-1}$,
        then 
        \begin{align*}
            \dist\left( h(x,s), h(y,t) \right) 
            \leq 
            \dist\left( (x,s), (y,t) \right) 
            \Lip\left( h, V \cap Q_{\Sigma,j} \setminus O_{\Sigma,j-1} \right) 
            .
        \end{align*}
        If no such $j$ exists, then there exists $j \in \bbN_{0}$ such that,
        without loss of generality, $(x,s)$ is below or on the graph of $\lambda_{\Sigma,j}$ 
        and $(y,t)$ is above or on that graph. 
        Now,
        \begin{align*}
            \dist\left( h(x,s), h(y,t) \right) 
            &\leq  
            \dist\left( h(x,s), h(x,\lambda_{\Sigma,j}(x)) \right) 
            \\&\qquad
            +
            \dist\left( h(x,\lambda_{\Sigma,j}(x)), h(y,\lambda_{\Sigma,j}(y)) \right) 
            \\&\qquad
            + 
            \dist\left( h(y,\lambda_{\Sigma,j}(y)), h(y,t) \right) 
            .
        \end{align*}
        The straight line from $(x,s)$ to $(x,\lambda_{\Sigma,j}(x))$ lies within $V$ 
        and crosses only finitely many graphs of the partition of unity,
        and the same is true for the straight line segment from $(y,t)$ to $(y,\lambda_{\Sigma,j}(y))$.
        Hence
        \begin{align*}
            \dist\left( h(x,s), h(x,\lambda_{\Sigma,j}(x)) \right) 
            \leq 
            \left| \lambda_{\Sigma,j}(x) - s \right|
            \max_{ 1 \leq l \leq j } 
            \Lip\left( h, V \cap Q_{\Sigma,l} \setminus O_{\Sigma,l-1} \right)
            \\
            \dist\left( h(y,\lambda_{\Sigma,j}(y)), h(y,t) \right) 
            \leq 
            \left| t - \lambda_{\Sigma,j}(y) \right|
            \max_{ j \leq l \leq n } 
            \Lip\left( h, V \cap Q_{\Sigma,l} \setminus O_{\Sigma,l-1} \right) 
            .
        \end{align*}
        We notice that  
        \begin{align*}
            &
            \left| s - \lambda_{\Sigma,j}(x) \right| + \left| \lambda_{\Sigma,j}(y) - t \right|
            \leq 
            \left| s - t \right| + \left| \lambda_{\Sigma,j}(x) - \lambda_{\Sigma,j}(y) \right|
            .
        \end{align*}
        Since $(x,\lambda_{\Sigma,j}(x))$ and $(y,\lambda_{\Sigma,j}(y))$ are on the graph of $\lambda_{\Sigma,j}$,
        we easily estimate 
        \begin{align*}
            &
            \dist\left( h(x,\lambda_{\Sigma,j}(x)), h(y,\lambda_{\Sigma,j}(y)) \right)
            \\&\; 
            \leq 
            \dist\left( (x,\lambda_{\Sigma,j}(x)), (y,\lambda_{\Sigma,j}(y)) \right)
            \Lip\left( h_{j}, V \cap Q_{\Sigma,j} \setminus O_{\Sigma,j-1} \right)
            ,
        \end{align*}
        and subsequently take the upper bound 
        \begin{align*}
            \dist\left( (x,\lambda_{\Sigma,j}(x)), (y,\lambda_{\Sigma,j}(y)) \right)
            &
            \leq 
            \dist\left( x, y \right)
            +
            \left| \lambda_{\Sigma,j}(x) - \lambda_{\Sigma,j}(y) \right|
            \\&
            \leq 
            \left( 1 + \Lip( \lambda_{\Sigma,j}, \Ball(z,\beta) ) \right)
            \dist\left( x, y \right)
            .
        \end{align*}
We have thus estimated the local Lipschitz constant of $h$ itself.

        \item 
        We want to show that the inverse of $h$ (from its image) is locally Lipschitz or even Lipschitz, 
        Recall that we assume $V$ to be in the hypergraph of $\lambda_{\Sigma,m}$ and the hypograph of $\lambda_{\Sigma,n}$, as explained above. 
        
        For $k, l \in \bbN_{0}$ with $k < l$ we define 
        \begin{align*}
            Q_{\Sigma,k,l} 
            := 
            \left\{ 
                (x,t) \in B \times [0,1]
            \suchthat* 
                x \in \Ball(z,\beta) \cap {S}, \; \lambda_{\Sigma,k}(x) \leq t \leq \lambda_{\Sigma,l}(x) 
            \right\}
            .
        \end{align*}
        It suffices to show that $h$ is a LIP embedding over $Q_{\Sigma,m,n}$,
        because our assumptions on $m$ and $n$ already mean that 
        \begin{align*}
            V \subseteq Q_{\Sigma,m,n}.
        \end{align*}
        We define mappings $\Theta_{k,l} : Q_{\Sigma,k,l} \to X$ via
        \begin{align*}
            \Theta_{k,l}(x,t) := \left( g_{1} g_{2} \dots g_{k} \right)^{\inv} \circ h(x,t),
            \quad
            (x,t) \in Q_{\Sigma,k,l}.
        \end{align*}
        In other words, 
        \begin{align*}
            \Theta_{k,l}(x,t) := g_{k+1} \dots g_{j-2} g_{j-1} \circ c_{j} q_{j}(x,t),
            \quad 
            (x,t) \in Q_{\Sigma,j} \setminus O_{\Sigma,j-1},
            \quad 
            k < j \leq l.
        \end{align*}
        In particular, $\Theta_{k,l}$ is an embedding.

        We develop a recursive estimate for the mappings $\Theta_{k,l}$.
        We start with observing that all $(x,s), (y,t) \in Q_{\Sigma,n-1,n}$ satisfy 
        \begin{align*}
            \dist\left( \Theta_{n-1,n}(x,s), \Theta_{n-1,n}(y,t) \right)
            &
            = 
            \dist\left( c_{n} {} q_{n}(x,s), c_{n} {} q_{n}(y,t) \right)
            \\&
            \geq 
            \LipInv( c_{n} {} q_{n}, V \cap Q_{\Sigma,n-1,n} )^{-1}
            \dist\left( (x,s), (y,t) \right)
            \\&
            \geq 
            \LipInv( c_{n} {} q_{n}, V \cap Q_{\Sigma,n} \setminus O_{\Sigma,n-1} )^{-1}
            \dist\left( (x,s), (y,t) \right)
            .
        \end{align*}
        Let now $j \in \bbN$ and $m < j < n$. 
On the one hand, for all $(x,s), (y,t) \in Q_{\Sigma,j,n}$, 
        \begin{align*}
\dist\left( \Theta_{j-1,n}(x,s), \Theta_{j-1,n}(y,t) \right)
&
            =
            \dist\left( g_{j} {} \Theta_{j,n}(x,s), g_{j} {} \Theta_{j,n}(y,t) \right)
            \\&\; 
            \geq 
            \LipInv( g_{j} \Theta_{j,n}, V \cap Q_{\Sigma,j,n} )^{-1}
            \dist\left( (x,s), (y,t) \right)
.
        \end{align*}
        On the other hand, for all $(x,s), (y,t) \in Q_{\Sigma,j-1,j}$, 
        \begin{align*}
            \dist\left( \Theta_{j-1,n}(x,s), \Theta_{j-1,n}(y,t) \right)
            &=
            \dist\left( c_{j} {} q_{j}(x,s), c_{j} {} q_{j}(y,t) \right)
            \\&
            \geq 
            \LipInv( c_{j} {} q_{j}, V \cap Q_{\Sigma,j-1,j} )^{-1}
            \dist\left( (x,s), (y,t) \right)
            \\&
            \geq 
            \LipInv( c_{j} {} q_{j}, V \cap Q_{\Sigma,j} \setminus O_{\Sigma,j-1} )^{-1}
            \dist\left( (x,s), (y,t) \right)
            .
        \end{align*}
        Finally, 
        if $(x,s) \in Q_{\Sigma,j,n}$ 
        and $(y,t) \in Q_{\Sigma,j-1,j}$,
        then 
        \begin{align*}
            \dist\left( \Theta_{j-1,n}(x,s), \Theta_{j-1,n}(y,t) \right)
            &=
            \dist\left( g_{j} {} \Theta_{j,n}(x,s), c_{j} {} q_{j}(y,t) \right)
            ,
            \\
            g_{j} {} \Theta_{j,n}(y,\lambda_{\Sigma,j}(y))
            &=
            c_{j} {} q_{j}(y,\lambda_{\Sigma,j}(y))
            ,
        \end{align*}
        and 
        \begin{align*}
            &
            \vartheta_{j}
            \dist\left( g_{j} {} \Theta_{j,n}(x,s), c_{j} {} q_{j}(y,t) \right)
            \\
            &\;
            \geq 
            \dist\left( g_{j} {} \Theta_{j,n}(x,s), g_{j} {} \Theta_{j,n}(y,\lambda_{\Sigma,j}(y)) \right)
            +
            \dist\left( c_{j} {} q_{j}(y,\lambda_{\Sigma,j}(y)), c_{j} {} q_{j}(y,t) \right)
.
        \end{align*}
        We see that the other two cases mentioned above deliver the required estimate. 
In summary, we have shown the recursive estimate 
        \begin{align*}
            &
            \vartheta_{j} 
            \dist\left( \Theta_{j-1,n}(x,s), \Theta_{j-1,n}(y,t) \right)
            \\&\; 
            \geq 
            \max\left( 
                \LipInv( g_{j} {} \Theta_{j,n}, V \cap Q_{\Sigma,j,n} )
                ,
                \LipInv( c_{j} {} q_{j}, V \cap Q_{\Sigma,j,j-1} )
            \right)^{-1}
            \dist\left( (x,s), (y,t) \right)
            .
        \end{align*}
        After a finite number of iterations, 
        the inequalities show that $\Theta_{m,n} : Q_{\Sigma,m,n} \to X$ is a LIP embedding. 
        Hence $h$ restricted to $Q_{\Sigma,m,n}$ is a LIP embedding, 
        and so is the restriction of $h$ to the open set $V \subseteq Q_{\Sigma,m,n}$.
        In particular, $h$ restricted $Q_{\Sigma,0,i} = Q_{\Sigma,i}$, which equals $h_{i}$, is a LIP embedding. 
        Consequently, $h$ is a LIP embedding. 
         
        \item 
        Consider the special case that the local collars $c_{i}$ are uniformly bi-Lipschitz
        and the partition of unity functions $\lambda_{i}$ are uniformly Lipschitz. 
        Using the definition of $h$, 
        we easily estimate 
        \begin{align*}
            \Lip   \left( h, V \cap Q_{\Sigma,i} \setminus O_{\Sigma,i-1} \right) 
&
            \leq 
            \Lip\left( c_i \right)
            \Lip\left( q_i \right)
            \left( \max_{ j \leq i } \Lip   \left( g_j \right) \right)^{N}
            ,
            \\ 
            \LipInv\left( h, V \cap Q_{\Sigma,i} \setminus O_{\Sigma,i-1} \right) 
&
            \leq  
            \LipInv\left( c_i \right)
            \LipInv\left( q_i \right)
            \left( \max_{ j \leq i } \LipInv\left( g_j \right) \right)^{N}
            . 
        \end{align*}
        Then we can choose $\beta = \infty$ and $I = [0,1]$,
        and thus can assume $\Ball(z,\rho) = B \times [0,1]$ in the discussion above.
        Given $(x,s), (y,t) \in B \times [0,1]$, 
        we pick the open neighborhoods ${S}_x$ and ${S}_y$ of $x$ and $y$, respectively, and set ${S} := {S}_x \cup {S}_y$.

        We now verify the Lipschitz inequalities for the pair $(x,s)$ and $(y,t)$,
        applying the uniform upper bounds and some elementary estimates. 
        We conclude that $c$ is bi-Lipschitz over $B \times [0,1]$.
    \end{enumerate}
    The proof is complete. 
\end{proof}

\begin{remark}
    A continuous function that is Lipschitz on two subsets of a metric space need not be Lipschitz on the union of the two subspaces. 
    The preceding proof has implicitly used the proof idea of Theorem~2.22 in~\cite[Theorem~2.22]{luukkainen1977elements}.
\end{remark}
\begin{remark}
    Theorem~\ref{theorem:lipschitz:locallyfinitecountable} assumes that the collar collection is countable and therefore extends Theorem~\ref{theorem:collaring:locallyfinitecountable}.
    Alternatively, we could have extended Theorem~\ref{theorem:collaring:locallyfinitegeneralized},
    no longer assuming that the collar collection is countable 
    but that the family of images of the local collars satisfies a weaker form of star-finiteness.
    The resulting differences in the proof are mostly technical. 
    The additional conditions for the global collar to be bi-Lipschitz remain the same. 
    Lastly, if the subspace $B$ is compact, then we always assume that collar cover has only finitely many, say, $N$ distinct non-empty members.
\end{remark}

\begin{theorem}[LIP collar theorem]\label{theorem:lipschitz:main}
    Let $B$ be a closed subset of a metric space $X$.
    If $B$ is locally LIP collared, then $B$ is LIP collared.
\end{theorem}
\begin{proof}
    Since $B$ is locally LIP collared in $X$,
    there exists a collection of sets $\calU = ( U_{\alpha} )_{\alpha \in \kappa}$ 
    and a collection of local LIP collars 
    \[
        c_\alpha : \overline{U_\alpha} \times [0,1] \to X, \quad \alpha \in \kappa. 
    \]
    Since $B$ is paracompact in $X$,
    there exists a locally finite cover $\calU = \bigcup_{i\in\bbN} \calU_{i}$,
    where $\calU_{i} = ( U_{i,\alpha} )_{\alpha \in \kappa}$ is a relatively open family such that
    for all $i \in \bbN$ and $\alpha \in \kappa$ we have the inclusion $U_{i,\alpha} \subseteq U_\alpha$. 
We set $U_{i} := \cup_{\alpha \in \kappa} U_{i,\alpha}$.
    Notably, $\overline{ U_{i} } = \cup_{\alpha\in\kappa} \overline{U_{i,\alpha}}$. 
    We now use Lemma~\ref{lemma:lipschitz:collaringdiscrete} to show the existence of local LIP collars 
    \begin{align*}
        c_{i} : \overline{U_{i}} \times [0,1] \to X.
    \end{align*}
    Moreover, the family $(U_{i})_{i \in \bbN}$ is locally finite and countable. 
    Using Theorem~\ref{theorem:lipschitz:locallyfinitecountable},
    we conclude that $B$ is LIP collared in $X$.
\end{proof}

\begin{remark}
    Luukkainen and V\"ais\"al\"a~\cite[Definition~7.2]{luukkainen1977elements} use a slightly different definition of collars: 
    given any subset $B$ of a metric space $X$, 
    they define a collar of $B$ in $X$ is a LIP embedding $c : B \times [0,1) \to X$ whose image is an open neighborhood of $B$ 
    and which satisfies $c(x,0)=x$ for all $x \in B$.
    They do not require the base set $B$ to be closed, 
    and therefore their LIP version of the collar theorem is not entirely included in our statement.
\end{remark}

\begin{remark}
    We remark on the additional conditions in Theorem~\ref{theorem:lipschitz:locallyfinitecountable}
    that enable estimates of the global bi-Lipschitz constants. 
    It is intuitive that we need uniform estimates on the local collars $c_{i}$.
    The uniform estimates on the Lipschitz constants in the partition of unity 
    reflect the requirement that the bases of the local collars are ``uniformly well-shaped''. 
    The existence of the constant $\zeta_{i}$ reflects a global property of the collection of local collars:
    parts of the metric space, even the image of other local collars, might be arbitrary close to any point within a local collar
    while being disjoint from that local collar. 
    Recall that in a metric space, balls of any radius may not even be connected.
    Lastly, the constant $N$ reflects a uniform local finiteness of the collar collection. 
\end{remark}

\begin{remark}
    We have described how to construct a LIP collar starting from any collection of LIP collars 
    whose bases constitute a cover of $B$. 
    By contrast, the LIP collar estimates are only available for countable collections of collars 
    that satisfy uniform estimates. It is an intriguing question under what conditions, if any, 
    the proof can be generalized to arbitrary collar covers that satisfy uniform estimates. 
\end{remark}

Having discussed one-sided collars, we address the construction of LIP bicollars. 
We begin with the following observation.

\begin{lemma}
    Let $B$ be a closed subset of a metric space $X$ that is two-sided in an open neighborhood $S$. 
    If $B$ is locally LIP bicollared, then $B$ is closed and locally LIP collared in $B \cup S^{+}$ and $B \cup S^{-}$.
\end{lemma}
\begin{proof} 
    This is the same proof as for Lemma~\ref{lemma:twosidedlocallybicollaredimplieslocallycollared},
    except that we invoke Lemma~\ref{lemma:lipschitz:cutfunctionthatavoidsclosedset} instead of Corollary~\ref{corollary:restrictedcollarthatavoidsclosedset}. 
\end{proof}

We proceed with an as-for-now partial result, which discusses the existence of LIP bicollars and an explicit bound for the global Lipschitz constant of the embedding. Analyzing when the collar is not only Lipschitz but even bi-Lipschitz will be accomplished below, and we even provide different means of the completing the estimates.

\begin{theorem}\label{theorem:lipschitz:bicollar}
    Let $B$ be a closed subset of a metric space $X$.
    Let 
    \begin{align*}
        \calC = \left( c_{\alpha} : \overline{U_{\alpha}} \times [-1,1] \to X \right)_{\alpha \in \kappa} 
    \end{align*}
    be a collection of LIP local bicollars of $B$ in $X$
    such that the family $\calU = (U_{\alpha})_{\alpha \in \kappa}$ covers $B$.

    \begin{enumerate}[1., wide=10pt, itemindent=\parindent, leftmargin=0pt, topsep=0pt, itemsep=0pt]
    \item 
    If $c : B \times [-1,1] \to X$ is a bicollar of $B$ in $X$ 
    whose restrictions to $B \times [-1,0]$ and $B \times [0,1]$ are LIP embeddings, 
    then $c$ is a LIP embedding.
    \item 
    If the restrictions of $c$ to $B \times [-1,0]$ and $B \times [0,1]$ are Lipschitz, then 
    \begin{align*}
        \Lip(c, B \times [-1,1] )
        \leq 
        \max\left(
            \Lip(c, B \times [-1,0] )
            ,
            \Lip(c, B \times [ 0,1] )
        \right)
        .
    \end{align*}    
    \end{enumerate}
\end{theorem}

\begin{proof}
    In what follows, we use 
    \begin{gather*}
        c^{+} : B \times [0,1]  \to X, \quad (x,t) \mapsto c(x,t),
        \\
        c^{-} : B \times [-1,0] \to X, \quad (x,t) \mapsto c(x,t).
    \end{gather*}
    We show the result in several steps. 
    \begin{enumerate}[1., wide=10pt, itemindent=\parindent, leftmargin=0pt, topsep=0pt, itemsep=0pt]
        \item 
        By construction, $c$ is a LIP embedding over $B \times ([-1,0)\cup(0,1])$.
Let $z \in B$.
There exist relatively open neighborhoods $U^{+}, U^{-} \subseteq B$ of $z$ and $t^{+} \in (0,1]$, $t^{-} \in [-1,0)$ such that 
        $c^{+}$ is bi-Lipschitz over $U^{+} \times [0,t^{+}]$
        and 
        $c^{-}$ is bi-Lipschitz over $U^{-} \times [t^{-},0]$.
        
        Explicitly, 
        if $s,t \in [0,t^{+}]$, then 
        \begin{align*}
            \dist( (x,s), (y,t) ) 
            &\leq 
            \LipInv(c^{+},U^{+} \times [0,t^{+}]) 
            \dist( c(x,s), c(y,t) ) 
            ,
            \\
            \dist( c(x,s), c(y,t) ) 
            &\leq 
            \Lip(c^{+},U^{+} \times [0,t^{+}])
            \dist( (x,s), (y,t) ) 
            ,
        \end{align*}
        and if $s,t \in [t^{-},0]$, then 
        \begin{align*}
            \dist( (x,s), (y,t) ) 
            &
            \leq 
            \LipInv(c^{-},U^{-} \times [t^{-},0]) 
            \dist( c(x,s), c(y,t) ) 
            ,
            \\
            \dist( c(x,s), c(y,t) ) 
            &
            \leq 
            \Lip(c^{-},U^{-} \times [t^{-},0])
            \dist( (x,s), (y,t) ) 
            .
        \end{align*}
        If $s < 0 < t$, then 
        \begin{align*}
            \dist( c(x,s), c(y,t) )
            &
            \leq 
            \dist( c(x,s), x )
            +
            \dist( x, y )
            +
            \dist( y, c(y,t) )
            \\&
            \leq 
            \max(
                \Lip(c^{+},U^{+} \times [0,t^{+}])
                ,
                \Lip(c^{-},U^{-} \times [t^{-},0])
            )
            \left( |s| + \dist( x, y ) + |t| \right)
            \\&
            \leq 
            \max(
                \Lip(c^{+},U^{+} \times [0,t^{+}])
                ,
                \Lip(c^{-},U^{-} \times [t^{-},0])
            )
            \dist( (x,s), (y,t) )
            .
        \end{align*}
        This shows that $c$ is locally Lipschitz:
        \begin{align*}
            \Lip(c, (U^{+} \cap U^{-}) \times (t^{-},t^{+}) )
            \leq 
            \max(
                \Lip(c^{+},U^{+} \times [0,t^{+}])
                ,
                \Lip(c^{-},U^{-} \times [t^{-},0])
            )
            .
        \end{align*}
        \item 
        In particular, if $c^{+}$ and $c^{-}$ are both Lipschitz, 
        then we can choose $U^{+} = U^{-} = B$ and $t^{-} = -1$, $t^{+}=1$,
        and obtain a Lipschitz estimate for $c$ itself.
        
        \item 
        We now address the converse estimate. We need to show that $c$ is a LIP embedding.
        We abbreviate \[V_{0} := (U^{+} \cap U^{-}) \times [t^{-},t^{+}].\]
        There exists $\delta > 0$ small enough such that $\Ball(z,\delta) \subseteq U_{\alpha}$ for some $\alpha \in \kappa$.
        We write 
        \begin{align*}
            U^{\star} := U^{+} \cap U^{-} \cap \Ball(z,\delta).
        \end{align*}
        Going further, 
        the fact that $c$ is continuous suffices to establish that there exists $\delta > 0$ and $\epsilon > 0$ so small that
        \begin{align*}
            E^{\star} := c( U^{\star} \times (-\epsilon,\epsilon) ) \subseteq c_\alpha( U_\alpha \times (-1,1) ).
        \end{align*}
        Now $U^{\star} \times (-\epsilon,\epsilon)$ is an open neighborhood of $(z,0)$ in $B \times [-1,1]$,
        and $E^{\star}$ is an open neighborhood of $z$ in $c_\alpha( U_\alpha \times (-1,1) )$. 
        Once more using the continuity of $c$, 
        we can choose $\delta$ and $\epsilon$ small enough to ensure that $c_{\alpha}$ is bi-Lipschitz over $c_\alpha^{-1}(E^{\star})$.

        Let $(x,s), (y,t) \in U^{\star} \times (-\epsilon,\epsilon)$ with $s < 0 < t$. 
        There exist $x', y' \in U^{\star}$ and $s' \in [-1,0]$ and $t' \in [0,1]$ such that 
        \begin{align*}
            c_\alpha(x',s') = c(x,s), 
            \quad 
            c_\alpha(y',t') = c(y,t).
        \end{align*}
        We observe 
        \begin{align*}
            &
            \dist( c_{\alpha}( x',s'), c_{\alpha}(y',t') )
            \geq 
            \LipInv( c_{\alpha}, c_{\alpha}^{-1} E^{\star} )^{-1}
            \left( \dist( x', y' ) + |s'| + |t'| \right)
            ,
        \end{align*}
        and then 
        \begin{align*}
            &\left( \dist( x', y' ) + |s'| + |t'| \right)
            \\\quad 
            &\geq 
            \Lip( c_{\alpha}, c_{\alpha}^{-1} E^{\star} )^{-1}
            \left( 
                \dist( c_{\alpha}( x',s'), x' )
                +
                \dist( x', y' ) 
                +
                \dist( y', c_{\alpha}(y',t') )
            \right)
            \\\quad 
            &\geq 
            \Lip( c_{\alpha}, c_{\alpha}^{-1} E^{\star} )^{-1}
            \left( 
                \dist( c( x,s), x' )
                +
                \dist( x', y' ) 
                +
                \dist( y', c(y,t) )
            \right).
        \end{align*}
        Next, we notice that 
        \begin{align*}
            \dist( c( x,s), x' )
            &\geq 
            \LipInv( c^{-}, U^{\star} \times [-\epsilon,0] )^{-1}
            \left( |s| + \dist(x,x') \right)
            ,
            \\
            \dist( c( y,t), y' )
            &\geq 
            \LipInv( c^{+}, U^{\star} \times [0,\epsilon] )^{-1}
            \left( |t| + \dist(y,y') \right)
            .
        \end{align*}
        The triangle inequality shows 
        \begin{align*}
            |s| + \dist(x,x') 
            +
            \dist(x',y')
            +
            |t| + \dist(y,y') 
            \geq 
            |s-t| + \dist(x,y) 
            .
        \end{align*}
        The combination with previous estimates 
        shows that $c$ is bi-Lipschitz over an open neighborhood of $(z,0)$. 
        In particular, it follows that $c$ is a LIP embedding.
    \end{enumerate}
    The proof is complete. 
\end{proof}

The preceding result is only partial, since we have not proven that the collar is bi-Lipschitz. 
One possibility to construct a bi-Lipschitz collar from two one-sided bi-Lipschitz collars 
uses a restriction of the original bicollar. 

\begin{theorem}\label{theorem:lipschitz:bicollar:complete}
    In addition to the assumptions of Theorem~\ref{theorem:lipschitz:bicollar},
    assume that $\calU$ has Lebesgue number $\delta \in (0,1]$, 
    and that the restrictions of $c$ to $B \times [-1,0]$ and $B \times [0,1]$ are bi-Lipschitz. 
    If $\epsilon \in (0,1]$ is chosen so small that it satisfies 
\begin{align*}
        \Lip(c) \max_{\alpha \in \kappa}\LipInv( c_\alpha ) \cdot \epsilon \leq \delta,
        \quad 
        4\Lip(c) \cdot \epsilon \leq \delta 
        ,
    \end{align*}
    then $c : B \times [-\epsilon,\epsilon] \to X$ is bi-Lipschitz with 
    \begin{align*}
        \LipInv( c, B \times [-\epsilon,\epsilon] )
        \leq 
        2 \left( 1 + 2 \epsilon \delta^{-1} \right)
        \max\left( 
            \LipInv( c^{-} )
            ,
            \LipInv( c^{+} )
        \right)
        \max_{\alpha \in \kappa} \left( \biLip( c_{\alpha} )^{-1} \right)
        .
    \end{align*} 
\end{theorem}
\begin{proof}
    We continue the discussion in the proof of Theorem~\ref{theorem:lipschitz:bicollar}.
    
    \begin{enumerate}[1., wide=10pt, itemindent=\parindent, leftmargin=0pt, topsep=0pt, itemsep=0pt]
    \item 
    We construct a bi-Lipschitz collar by restricting the Lipschitz bicollar. 
    Let $x,y \in B$ and $s,t \in [-1,1]$.
    Clearly, if $s, t \geq 0$ or $s,t \leq 0$, 
    then 
    \begin{align*}
        |s-t| + \dist(x,y) 
        \leq 
        \max\left( 
            \LipInv( c^{-} )
            ,
            \LipInv( c^{+} )
        \right) 
        \dist( c( x,s), c_{\alpha}(y,t) ) .
    \end{align*}
    It remains to consider the case $s < 0 < t$. 
    
    \item 
    We first make some geometric observations. 
    On the one hand, if $z \in U_\alpha$ such that $\Ball(z,\delta) \subseteq U_\alpha$, 
    then $c_\alpha$ maps the ball $\Ball((z,0),\delta)$ 
    onto a superset of the ball $\Ball( z, \delta / \LipInv(c_\alpha) )$.
    On the other hand, for any $\epsilon \in (0,1]$ we already know 
    \begin{gather*}
        c( B \times [-\epsilon,\epsilon] ) \subseteq \Ball(B, \Lip(c) \epsilon ).
\end{gather*}
    Hence $c( B \times [\epsilon,\epsilon] )$ is a subset of the union of the images of the local collars if
    \begin{align*}
        \Lip(c) \max_{\alpha \in \kappa}\LipInv( c_\alpha ) \cdot \epsilon \leq \delta 
        .
    \end{align*}        
    Now, if $\dist(x,y) \leq \delta$, then we are in a situation as already discussed in the proof of Theorem~\ref{theorem:lipschitz:bicollar}. 
    The computations lead to the estimate 
    \begin{align*}
        |s-t| + \dist(x,y) 
        \leq 
        \biLip( c_{\alpha} )^{-1}
        \max\left( 
            1
            ,
            \LipInv( c^{-} )
            ,
            \LipInv( c^{+} )
        \right)
        \dist( c( x,s), c_{\alpha}(y,t) )
        .
    \end{align*}
    If $\dist(x,y) \geq \delta$ and $|s|, |t| \in [0,\epsilon]$, then 
$|s-t| \leq 2 \epsilon \delta^{-1} \dist(x,y)$
and 
    \begin{align*}
        \dist(x,y)
        &
        \leq
        \dist\left( c(x,s), c(y,t) \right) + \dist\left( c(x,s), x \right) + \dist\left( c(y,t), y \right)
        \\&
        \leq
        \dist\left( c(x,s), c(y,t) \right) + \Lip(c^{-}) |s| + \Lip(c^{+}) |t|
        \\&
        \leq
        \dist\left( c(x,s), c(y,t) \right) + 2 \Lip(c) \epsilon \delta^{-1} \dist(x,y)
        .
    \end{align*}
    Consequently, if we choose $\epsilon > 0$ small enough so that $4 \Lip(c) \epsilon \leq  \delta$, then 
    \begin{align*}
\dist(x,y)
        \leq
        2
        \dist\left( c(x,s), c(y,t) \right) 
        .
    \end{align*}
    \end{enumerate}
    This leads to the desired estimate, and so $c$ is bi-Lipschitz over $B \times [-\epsilon, \epsilon]$. 
\end{proof}

The preceding discussion radically simplifies if the metric space features more structure. For example, the following result is applicable if the metric equals the induced path metric of the metric space, that is, $X$ is a so-called \emph{length space}.

\begin{theorem}\label{theorem:lipschitz:bicollar:midpoint}
    In addition to the assumptions of Theorem~\ref{theorem:lipschitz:bicollar},
    assume that there exists $\alpha \geq 1$ 
    such that for all $x,y \in B$ and $s \in [-1,0]$ and $t \in [0,1]$,
    there exists $z \in B$ satisfying 
    \begin{align*}
        \dist( c(x,s), z ) + \dist( z, c(y,t) ) \leq \alpha \dist(c(x,s),c(y,t)).
    \end{align*}
    Then $c : B \times [-1,1] \to X$ is bi-Lipschitz, and 
    \begin{align*}
        \LipInv(c) \leq \alpha \max( \LipInv(c^{-}), \LipInv(c^{+}) ) 
        .
    \end{align*}
\end{theorem}
\begin{proof}
    Let $x,y \in B$ and $s \in [-1,0]$ and $t \in [0,1]$. Then 
\begin{align*}
        \dist(x,y) + |s-t|
        &\leq 
        \dist(x,z) + \dist(z,y) + |s| + |t|
        \\  
        &\leq 
        \LipInv(c^{-})
        \dist( c(x,s), z )
        + 
        \LipInv(c^{+})
        \dist( z, c(y,t) )
        .
    \end{align*}
    The claim is now evident. 
\end{proof}

\begin{remark}
    Given a radius $r > 0$, 
    the circle $B := \partial\calB(r,0) \subseteq \bbR^{2}$ is collared in the disk $X := \calB(r,0) \subseteq \bbR^{2}$.
    For example, a simple choice of collar is 
    \begin{align*}
        c : B \times [0,1] \to X, \quad (x,t) \mapsto \left( 1 - t/2 \right) x.
\end{align*}
    As $r$ goes to zero, $\Lip(c)$ remains uniformly bounded but $\LipInv(c) \geq \frac 1 r$.
    The reason is that the parameter range $t \in [0,1]$ is not comparable to the length scale of the collared set. 
    In this example with a compact base,
    this undesirable phenomenon is easily mitigated by rescaling $t$ to range over $[0,r]$,
    or alternatively by rescaling the base set.
    Generally, a notion of collar in metric spaces should have $t$ range over an adapted length scale. 
\end{remark}

The following example is relevant for the analysis of partial differential equations~\cite{licht2019smoothed}.
    Let $\Omega \subseteq \bbR^{n}$ be a domain.
    We call $\Omega$ a \emph{weakly Lipschitz domain} if for all $z \in \partial\Omega$
    there exist an open neighborhood $O_{z} \subseteq \bbR^{n}$ together with a bi-Lipschitz mapping 
$\phi_{z} : [-1,1]^{n} \to \overline{O_{z}}$
satisfying 
    \begin{gather*}
        \phi_{z}\left( [-1,1]^{n-1} \times [-1,0) \right) = \overline{O_{z}} \cap \Omega
        ,
        \quad 
        \phi_{z}\left( [-1,1]^{n-1} \times ( 0,1] \right) = \overline{O_{z}} \setminus \overline\Omega
        ,
        \\
        \phi_{z}\left( [-1,1]^{n-1} \times \{0\} \right)      = \overline{O_{z}} \cap \partial\Omega
        .
    \end{gather*}
    In other words, a domain is weakly Lipschitz if its boundary is a Lipschitz manifold of dimension $n-1$ that is locally flat in the Lipschitz category.
    Given $\phi_{z}$, one easily constructs a bi-Lipschitz local bicollar: we set $U_{z} = O_{z} \cap \partial\Omega$ and 
    \begin{align*}
        c_{z} : \overline{U_{z}} \times [-1,1] \to \bbR^{n}, \quad (x,t) \mapsto \phi_{z}\left( \phi_{z}^{-1}( x ) + t \vec{e}_n \right).
    \end{align*}
We see that a Lipschitz submanifold of codimension one is locally flat in the Lipschitz category if and only if it is bicollared in the Lipschitz category.

\begin{proposition}
    Let $\Omega \subseteq \bbR^{n}$ be a weakly Lipschitz domain. 
Moreover, suppose that there exists $r > 0$ and $C > 0$ such that for each $x \in \partial\Omega$
    there exists a local bicollar 
$c_{x} : U_{x} \times [-1,1] \to \bbR^{n}$
where $U_{x} = \overline{\calB(x,1)} \cap \partial\Omega$ and $1 \leq \Lip(c_{x}), \LipInv(c_{x}) \leq C$.
    Then $\partial\Omega$ has a bi-Lipschitz collar $h$ satisfying 
    \begin{align*}
        \Lip\left( h \right) 
        &
        \leq 
        C
        ( 1 + L_{\Sigma} / 2 )
        ( 1 + 2 L_{\Sigma} )
        \max\left( 1 + 2 L , 1 + \zeta^{-1} \right)^{N}
        C^{2N}
        ,
        \\ 
        \LipInv\left( h \right) 
        &\leq
        C
        ( 1 + L_{\Sigma} / 2 )
        C^{4N} 
        \max\left( 
            1 + \frac C \zeta + \frac L 2
        \right)^{N}
        \max\left( 3 + \frac  1 \zeta + 3 L \right)^{N} ,
    \end{align*}
    where 
    \[
        L \leq 2 (N'-1), 
        \quad 
        L_{\Sigma} \leq 2 N', 
        \quad
        N' \leq 5^{n}, 
        \quad 
        N \leq (8C+9)^{n}, 
        \quad 
        \zeta = 0.25 / C
        .
    \]
\end{proposition}
\begin{proof}
    The existence of a LIP bicollar follows by Theorem~\ref{theorem:lipschitz:bicollar}.
    For the bi-Lipschitz estimates, we construct a countable cover and control the relevant constants.
    \begin{enumerate}[1., wide=10pt, itemindent=\parindent, leftmargin=0pt, topsep=0pt, itemsep=0pt]

    \item 
    We develop an auxiliary argument. 
    Henceforth, a set $S \subseteq \bbR^{n}$ is called $\tau$-separated for some $\tau > 0$ provided that $\dist(x,y) \geq \tau$ for all distinct $x,y \in S$.
    
    Whenever $\calS$ is a family of $\tau$-separated subsets of $B$, 
    then $\bigcup \calS$ is a $\tau$-separated subset of $B$ that contains every member of $\calS$.
    According to Zorn's lemma, there exists a maximal $\tau$-separated subset $S \subseteq B$. 
    
    It follows that $B \subseteq \bigcup_{x \in S} \calB(x,\tau)$, as is seen by contradiction:
    If some $z \in B$ has distance at least $\tau$ from every member of $S$, then $S$ cannot be maximal. 
    It also follows that $S$ is countable: 
    For each $x \in S$ we pick $q(x) \in \Ball(x,0.1) \cap \bbQ^{n}$, using the axiom of choice.
    Clearly, the mapping $x \mapsto q(x)$ is injective from $S$ into $\bbQ^{n}$, 
    and so $S$ must be countable.

    \item 
    So there exists countable $0.5$-separated family $S := (x_i)_{i \in \bbN}$ 
    such that $B \subseteq \bigcup_{x \in S} \calB(x,0.5)$.
    First, for each $x \in \partial\Omega$ 
    we have $\calB(x,0.5) \subseteq \calB(x_i,1)$ for some $i \in \bbN$. 
    Second, given $x \in B$, the number of $y \in S$ for which $x \in \calB(y,1)$ is bounded by 
$N' := 5^{n}$.
    By Proposition~\ref{proposition:lipschitzpartitionofunity} with $N'$ as above, $\delta \equiv 0.5$, and $\delta_{0} = \delta/2$, 
    there exists a partition of unity $(\lambda_{i})_{i\in\bbN}$
    subordinate to $(\calB(x_i,1))_{i \in \bbN}$,
    whose Lipschitz constants over their supports are bounded by $L := 2 (N'-1)$,
    and such that $\calB(z,0.25) \subseteq U_{x_i}$ for each $z \in \supp(\lambda_{i})$, $i \in \bbN$.
    Moreover, arbitrary sums of that partition of unity have Lipschitz constant bounded by $L_{\Sigma} := 2 N'$.
    Thence, the ball around $c_{x_i}( \supp(\lambda_i) \times [0,0.75] )$ of radius $\zeta := 0.25 / C$ is in the image of $c_{x_i}$. 
    
    Finally, recall that the image of $c_{x}$ is included in $\calB(x,1+C)$ whenever $x \in B$. 
    For each $x \in S$, the number of those $y \in S$ for which $\calB(x,1+C) \cap \calB(y,1+C) \neq \emptyset$ is bounded by 
$N := (2C+2.25)^{n} (0.25)^{-n} = (8C+9)^{n}$.
\end{enumerate}
    In accordance with Theorem~\ref{theorem:lipschitz:locallyfinitecountable}, Theorem~\ref{theorem:lipschitz:bicollar},
    and Theorem~\ref{theorem:lipschitz:bicollar:midpoint}, there exists a bi-Lipschitz bicollar embedding $h$ satisfying the desired bounds. 
\end{proof}

\subsection*{Acknowledgement}
This material is based upon work supported by the National Science Foundation 
under Grant No.\ DMS-1439786 while the author was in residence at the Institute for Computational and 
Experimental Research in Mathematics in Providence, RI, during the ``Advances in Computational Relativity'' program.

\end{document}